\documentclass[11pt,reqno]{amsart}

\usepackage[utf8]{inputenc}
\usepackage{lmodern}
\usepackage[T1]{fontenc}
\usepackage{amsmath,amssymb}
\usepackage{hyperref}
\usepackage{mathrsfs}
\usepackage{amsfonts}
\usepackage{mathtools}
\usepackage{tikz}
\usepackage{lmodern}
\usepackage{bigints}
\usepackage{relsize}
\usepackage{bbm}
\usepackage{amsthm}
\usepackage{mathtools}
\mathtoolsset{showonlyrefs}
\usepackage{geometry}
\usepackage{scalerel,stackengine}
\stackMath
\newcommand\reallywidehat[1]{
\savestack{\tmpbox}{\stretchto{
  \scaleto{
    \scalerel*[\widthof{\ensuremath{#1}}]{\kern-.6pt\bigwedge\kern-.6pt}
    {\rule[-\textheight/2]{1ex}{\textheight}}WIDTH-LIMITED BIG WEDGE
  }{\textheight} 
}{0.5ex}}
\stackon[1pt]{#1}{\tmpbox}
}
\numberwithin{equation}{subsection}
\usepackage{upgreek}
\usepackage{wrapfig}
\usepackage{amsmath}
\usepackage{amsthm}
\usepackage{mathrsfs}
\usepackage{color}
\usepackage{xcolor}
\usepackage{graphicx}
\usepackage{amssymb}
\usepackage{ dsfont }
\usepackage{url}
\usepackage{multicol}
\usepackage{tikz}
\usepackage{amsmath}
\usepackage{fancyhdr}
\usetikzlibrary{matrix}
\usetikzlibrary{arrows}
\usepackage{subfig}
\usepackage{float}
\usepackage{hyperref}
\hypersetup{hidelinks}
\usepackage{listings}
\usepackage{esint}
\usepackage{enumitem}
\usepackage{varwidth}

\usepackage{upgreek}

\newcommand{\cgc}[1]{}

\allowdisplaybreaks

\title[]{On the internal modes of the ground state for 3D cubic--quintic equations}
\author[G. Chen]{Gong Chen}

\thanks{G.C. was partially supported by NSF Grants DMS-2350301 and CAREER-DMS-2540992, by the Simons Foundation MP-TSM-00002258, and by the AMS Stefan Bergman Fellowship.}

\email{gc@math.gatech.edu}

\author[Z. Yang]{Zhaojie Yang}
\email{zyang768@gatech.edu}

\address{School of Mathematics, Georgia Institute of Technology, 686 Cherry Street,
Atlanta, GA}
\newcommand{\be}{\begin{equation}}
\newcommand{\ee}{\end{equation}}
\newcommand{\bp}{\begin{proof}}
\newcommand{\ep}{\end{proof}}
\newcommand{\bel}{\begin{equation}\label}
\newcommand{\eeq}{\end{equation}}
\newcommand{\bea}{\begin{eqnarray}}
\newcommand{\eea}{\end{eqnarray}}
\newcommand{\bee}{\begin{eqnarray*}}
\newcommand{\eee}{\end{eqnarray*}}
\newcommand{\ben}{\begin{enumerate}}
\newcommand{\een}{\end{enumerate}}

\newcommand{\R}{\mathbb{R}}

\newcommand{\re}{\operatorname{Re}}

\newcommand{\im}{\operatorname{Im}}

\newtheorem{thm}{Theorem}[section]
\newtheorem{cor}[thm]{Corollary}
\newtheorem{lem}[thm]{Lemma}
\newtheorem{prop}[thm]{Proposition}

\theoremstyle{remark}
\newtheorem{rem}{Remark}[section]

\definecolor{codegreen}{rgb}{0,0.6,0}
\definecolor{codegray}{rgb}{0.5,0.5,0.5}
\definecolor{codepurple}{rgb}{0.58,0,0.82}
\definecolor{backcolour}{rgb}{0.95,0.95,0.92}

\lstdefinestyle{mystyle}{
backgroundcolor=\color{backcolour},
commentstyle=\color{codegreen},
keywordstyle=\color{magenta},
numberstyle=\tiny\color{codegray},
stringstyle=\color{codepurple},
basicstyle=\footnotesize,
breakatwhitespace=false,
breaklines=true,
captionpos=b,
keepspaces=true,
numbers=left,
numbersep=5pt,
showspaces=false,
showstringspaces=false,
showtabs=false,
tabsize=2
}
\numberwithin{equation}{section}

\usepackage{pgfplots}
\pgfplotsset{compat=newest}



\theoremstyle{definition}

\numberwithin{ej}{section}

\usepackage{comment}

\begin{document}


\begin{abstract}
Motivated by the problem of asymptotic stability for ground states of the Klein--Gordon and Schr\"odinger equations in three spatial dimensions in the presence of internal modes, we study the corresponding models with cubic--quintic nonlinearities. These equations arise in several physical contexts; in particular, the cubic--quintic Klein--Gordon equation appears naturally in the study of spin-$0$ particles in quantum field theory.

More precisely, let $Q_\omega$ be the positive radial ground state of the three-dimensional
cubic--quintic elliptic equation
\[
 -\Delta Q_\omega+\omega Q_\omega-Q_\omega^3+Q_\omega^5=0.
\]
We study two related, but logically distinct, spectral problems as
$\omega$ approaches the endpoint $3/16$ of the ground-state branch.  First,
we determine the complete discrete spectrum of the scalar Hessian
\[
 L_+=-\Delta+\omega-3Q_\omega^2+5Q_\omega^4.
\]
This operator is also the exact linearized spatial operator around the
static state $Q_\omega$ for the real scalar cubic--quintic Klein--Gordon
equation with mass parameter $\omega$.  We prove that the number of its
discrete angular-momentum sectors tends to infinity, and obtain sharp
asymptotics for their locations and multiplicities.  Second, for the
cubic--quintic nonlinear Schr\"odinger equation, we analyze the full
Hamiltonian matrix linearization and prove the existence of internal modes
in precisely an interval of angular-momentum sectors of length comparable
to $(3/16-\omega)^{-1}$.
These results provide a rigorous linear foundation for future nonlinear
stability and radiation-damping analysis in three-dimensional
cubic--quintic models.

\end{abstract}

\maketitle

\setcounter{tocdepth}{1}

\begin{quote}

\tableofcontents

\end{quote}

\medskip
\section{Introduction}
Solitary waves are fundamental coherent structures arising in nonlinear dispersive equations. Understanding the dynamics of perturbations around such waves is a central problem in the theory of nonlinear wave equations. A key ingredient in this analysis is the spectral structure of the linearized operator around the soliton. In particular, eigenvalues located in the spectral gap below the continuous spectrum correspond to internal modes, which generate localized oscillations and play an important role in the interaction between solitary waves and dispersive radiation.

\medskip

The presence of internal modes significantly complicates the long-time dynamics near solitary waves. In the seminal work \cite{soffer1999resonances}, Soffer and Weinstein developed a mechanism describing how internal modes interact with radiation through nonlinear resonances, leading to radiation damping governed by the Fermi Golden Rule. Subsequent works, for example \cite{buslaev2003asymptotic,tsai2002asymptotic,cuccagna2003asymptotic,cuccagna2022revisiting,cuccagna2014asymptotic,leger2021internal,kopylova2011asymptotic,martel2024asymptotic,delort2022long,lei2022energy,lei2023energy}, together with the survey \cite{cuccagna2020survey}, further developed this theory and established asymptotic stability results for solitary waves in a wide range of nonlinear dispersive equations. In these analyses, internal modes play a central role: they generate long-lived oscillatory components whose interaction with the continuous spectrum ultimately governs the decay of perturbations.

\medskip
Despite their importance, to the best of our knowledge there is currently no literature establishing asymptotic stability of solitons \emph{under the influence of internal modes} for a natural, translation-invariant, power-type dispersive model in three spatial dimensions. Our broader goal is to identify such models and to develop a rigorous asymptotic stability theory incorporating both internal modes and modulation dynamics.

Motivated by the physical observations in \cite{anderson1983variational} and the variational stability framework in \cite{shatah1983stable}, we expect this phenomenon to occur for the cubic--quintic nonlinear Klein--Gordon equation
\begin{equation}\label{eq:CQKG}
    \partial_{tt}u - \Delta u + u - |u|^2u + |u|^4u = 0,
    \qquad (t,x)\in\R\times\R^3,
\end{equation}
which appears in the study of spin-$0$ particles in quantum field theory; see, e.g., \cite{lee1981particle}.
It is also well known that the corresponding cubic--quintic nonlinear Schr\"odinger equation shares closely related solitary-wave and spectral features; see, for instance, \cite{killip2017solitons}. 
Accordingly, we also consider
\begin{equation}\label{eq:CQNLS}
    i\partial_t u = -\Delta u - |u|^2u + |u|^4u,
    \qquad (t,x)\in\R\times\R^3.
\end{equation}

Both \eqref{eq:CQKG} and \eqref{eq:CQNLS} admit standing or solitary waves. For \eqref{eq:CQKG}, one seeks standing waves of the form
\begin{equation*}
  u(t,x)=e^{i\Omega t}\phi_\Omega(x),
\end{equation*}
which leads to the stationary equation
\begin{equation*}
   -\Delta \phi_\Omega + (1-\Omega^2)\phi_\Omega - |\phi_\Omega|^2\phi_\Omega + |\phi_\Omega|^4\phi_\Omega = 0.
\end{equation*}
For \eqref{eq:CQNLS}, solitary waves take the form
\begin{equation*}
    u(t,x)=e^{i\omega t}Q_\omega(x),
\end{equation*}
where \(Q_\omega\) solves
\begin{equation}\label{eq:NLS-elliptic}
    -\Delta Q_\omega + \omega Q_\omega - |Q_\omega|^2Q_\omega + |Q_\omega|^4Q_\omega = 0.
\end{equation}
Note that   for any admissible parameter
$\omega$, the same profile $Q_\omega$ is a static solution of the real
scalar equation
\begin{equation}\label{eq:scalar-cqkg-effective-mass}
 \partial_{tt}w-\Delta w+\omega w-w^3+w^5=0.
\end{equation}
Writing $w=Q_\omega+z$ gives, at the linear level,
\begin{equation}\label{eq:scalar-kg-linearization}
 \partial_{tt}z+L_+z=0,
 \qquad
 L_+=-\Delta+\omega-3Q_\omega^2+5Q_\omega^4.
\end{equation}
Thus an eigenvalue $\lambda\in(0,\omega)$ of $L_+$ produces a localized
scalar Klein--Gordon oscillation with temporal frequency $\sqrt\lambda$
below the continuum threshold $\sqrt\omega$; a negative eigenvalue gives a
hyperbolic direction, and the zero eigenvalue is generated by translation
invariance.

From this point on, for simplicity, we use the notation associated with \eqref{eq:NLS-elliptic} and \eqref{eq:CQNLS}. It is known that for \(0<\omega<\frac{3}{16}\), \eqref{eq:NLS-elliptic} admits infinitely many nontrivial solutions; in particular, there exists a unique positive solution \(Q_\omega\), up to translations, called the ground state. This solution is radial and strictly decreasing. 
Moreover, there exists a frequency \(\omega_*\in(0,\frac{3}{16})\) such that the ground state is orbitally unstable for \(0<\omega\le \omega_*\) and orbitally stable for \(\omega_*<\omega<\frac{3}{16}\); see \cite{zhang2025monotonicity}.

\medskip
As a first step toward the nonlinear asymptotic stability problem, the purpose of this paper is to study the discrete spectrum of the linearized operator in the threshold regime, where the frequency approaches the endpoint of the soliton branch, $\omega=\frac{3}{16}-\epsilon$. In this limit the ground state develops a large spatial scale, and the associated linearized operator acquires an increasingly rich discrete spectrum.

\cgc{maybe we put a picture here}

\medskip
Heuristically, this phenomenon can be understood from the spatial structure of the soliton. As the threshold is approached, the ground-state profile remains close to a constant on a region whose radius grows like

\begin{equation*}
    R_\epsilon\sim \frac{1}{\epsilon}.
\end{equation*}
Here \(\epsilon\) denotes the distance between the frequency and the threshold value. In this regime, the linearized operator behaves roughly like a Schr\"odinger operator with an effective potential supported on a region of size \(R_\epsilon\). From a semiclassical perspective, such a potential well typically admits a number of bound states proportional to its spatial extent, suggesting that the number of internal modes should grow like \(R_\epsilon\).

\medskip

Our main result confirms this heuristic picture and provides a precise description of the discrete spectrum close to the threshold $\frac{3}{16}$.
More precisely, we study the linearized operator around this ground state:

For \eqref{eq:CQNLS}, let
\[
u(t,x)=e^{i\omega t}\bigl(Q_\omega(x)+v(t,x)\bigr),
\qquad
v=a+ib,
\]
where \(a,b\) are real-valued. Linearizing
\[
i\partial_tu=-\Delta u-|u|^2u+|u|^4u
\]
around \(Q_\omega\), we obtain the full matrix operator
\begin{equation*}
	\partial_t
	\begin{pmatrix}
		a\\ b
	\end{pmatrix}
	=
	\mathcal L_\omega
	\begin{pmatrix}
		a\\ b
	\end{pmatrix},
\end{equation*}
where
\[
\mathcal L_\omega=
\begin{pmatrix}
	0&L_-\\
	-L_+&0
\end{pmatrix},
\]
with
\begin{equation*}
	L_+
	=
	-\Delta+\omega-3Q_\omega^2+5Q_\omega^4,
	\qquad
	L_-
	=
	-\Delta+\omega-Q_\omega^2+Q_\omega^4 .
\end{equation*}

Our main results are the following two theorems.
\begin{thm}\label{theorem-scalar}
    Let \(\epsilon=\frac{3}{16}-\omega\). Then the discrete eigenvalues of \(L_+\) below \(\omega\) are
    \begin{align*}
\lambda_0<\lambda_1=0<\lambda_2< \cdots < \lambda_{\ell_\epsilon}< \omega,
    \end{align*}
    where $$\ell_\epsilon = \frac{3}{16}\epsilon^{-1}+\mathcal{O}(1).$$
    In addition, for each \(0\leq\ell\leq\ell_\epsilon\), the eigenspace associated with \(\lambda_\ell\) is \(2\ell+1\)-dimensional. Moreover, the following estimates hold:
    \begin{align*}
    \lambda_{\ell}-\lambda_{\ell-1} = \frac{32}{3} \ell \epsilon^2+ \mathcal{O}(\ell \epsilon^3), \quad \ell\ge 1,
\end{align*}
    and
    \begin{align*}
        \lambda_\ell= \frac{16}{3}(\ell^2+\ell-2)\epsilon^2 + \mathcal{O}((\ell+1)^2 \epsilon^3), \quad \ell\ge 0.
    \end{align*}
\end{thm}
This result gives accurate estimates for the number and locations of the discrete eigenvalues.  In particular, the number of scalar discrete modes grows without bound as the threshold is approached. Moreover, the gaps between adjacent eigenvalues shrink, and these eigenvalues become dense in $(0,\omega)$. Together with the matrix result below, this provides a natural translation-invariant dispersive PDE model in three spatial dimensions whose linearization possesses arbitrarily many internal modes. Such a regime is expected to be particularly relevant for future investigations of nonlinear dynamics near solitary waves, including radiation-damping mechanisms governed by the Fermi Golden Rule. The spectral description obtained here provides the linear foundation for the analysis of modulation systems involving multiple internal modes and their interaction with dispersive radiation in cubic--quintic models.

\medskip
For the matrix operator, we have the following theorem:
\begin{thm}\label{theorem-matrix}
	There exist constants \(0<c<C<\infty\), independent of sufficiently small \(\epsilon\), and an integer \(\widetilde\ell_0\) satisfying
	\(c\epsilon^{-1}\le \widetilde\ell_0\le C\epsilon^{-1}\), such
	that \(\mathcal L^{(\ell)}\) has at least one internal mode for
	\(0\leq\ell\leq\widetilde\ell_0\), whereas \(\mathcal L^{(\ell)}\) has no internal modes for \(\ell>\widetilde\ell_0\).
\end{thm}
Here \(\mathcal L^{(\ell)}\) denotes the restriction of \(\mathcal L\) to the \(\ell\)-th angular-momentum sector in the spherical harmonic decomposition. Again, the theorem implies that the number of internal modes of \(\mathcal L\) grows without bound as the threshold is approached. Unlike in the scalar case, \(\mathcal L^{(\ell)}\) may possess several internal modes within a fixed angular-momentum sector, as the proof in Section~4 shows.

\medskip

We briefly outline the framework of the proof. In Section~2, we first analyze the asymptotic behavior of the soliton \(Q_\omega\) by treating its radial ODE separately in three spatial regions. An a priori analysis indicates that \(Q_\omega\) exhibits distinct behavior in these regions. For \(r\lesssim\epsilon^{-1}\), the profile forms a broad plateau and remains exponentially close to \(Q_\omega(0)\); for \(r\gtrsim\epsilon^{-1}\), it has a far-field tail that decays exponentially to zero; and for \(r\approx\epsilon^{-1}\), it undergoes a transition between these two regimes. We obtain precise asymptotics for \(Q_\omega\) by constructing suitable approximate profiles and applying refined comparison arguments.

\medskip
To analyze the internal modes of the linearized operator, we first apply the spherical harmonic decomposition. A key technical step is to prove that, for each \(\ell\), the operator \(L_+^{(\ell)}\) has at most one eigenvalue below the essential spectrum. Using monotonicity with respect to \(\ell\), we reduce the main part of this question to the radial sector \(\ell=0\). Our new observation is that, after shifting the transition region, the limiting operator associated with \(L_+^{(0)}\) is identical to the linearized operator around the kink in the one-dimensional \(\phi^6\) model, which is known to have exactly one eigenvalue. A careful perturbation argument then transfers this spectral property to \(L_+^{(0)}\). Finally, to determine the accurate number of internal modes and to locate these eigenvalues, we combine the asymptotics of \(Q_\omega\) obtained in Section~2 with carefully chosen trial functions and min--max comparisons.

\medskip
The rest of the paper is organized as follows. In Section~2, we derive detailed asymptotics for the ground-state profile in the threshold regime. In Section~3, we analyze the associated scalar linearized operator and prove Theorem~\ref{theorem-scalar}. In Section~4, we study the matrix linearized operator and prove Theorem~\ref{theorem-matrix}.

\section{Asymptotics of $Q_\omega$}
In this section, we study the asymptotics of $Q_\omega$ when $\omega\to \frac{3}{16}$. Define $\epsilon= \frac{3}{16}-\omega>0$, and
\begin{align*}
    &F_\epsilon(x)= \frac{1}{2}\left(\frac{3}{16}-\epsilon\right)x^2-\frac{1}{4}x^4+\frac{1}{6}x^6,\\
    &F_\epsilon'(x)=\left(\frac{3}{16}-\epsilon\right)x-x^3+x^5,\\
    &F_\epsilon''(x)=\left(\frac{3}{16}-\epsilon\right)-3x^2+5x^4.
\end{align*}
Define
\begin{align*}
    &Q_\epsilon^*= \sqrt{\frac{1}{2} + \sqrt{\frac{1}{16}+ \epsilon}},\\
    &\alpha=\sqrt{\frac{3}{4}-\sqrt{3\epsilon}},
\end{align*}
so that $F_\epsilon'(Q_\epsilon^*)=0$ and $F_\epsilon(\alpha)=0$, with
\begin{align*}
    \alpha < \sqrt{\frac{3}{4}} < Q_\epsilon^*, \quad \alpha=\sqrt{\frac{3}{4}} + \mathcal{O}(\sqrt{\epsilon}), \quad Q_\epsilon^*=\sqrt{\frac{3}{4}} + \mathcal{O}(\epsilon).
\end{align*}

For later reference, the relevant sign information can be read off
explicitly.  If
\[
 \beta_\epsilon^2:=\frac34+\sqrt{3\epsilon},
 \qquad
 q_{-,\epsilon}^2:=\frac12-\sqrt{\frac1{16}+\epsilon},
\]
then
\begin{align}
 F_\epsilon(q)
 &=\frac{q^2}{6}(q^2-\alpha^2)(q^2-\beta_\epsilon^2),
 \label{eq:F-factorization}\\
 F_\epsilon'(q)
 &=q(q^2-q_{-,\epsilon}^2)(q^2-(Q_\epsilon^*)^2).
 \label{eq:Fprime-factorization}
\end{align}
Moreover, writing $s_\epsilon=(Q_\epsilon^*)^2$ and using
$\omega=s_\epsilon-s_\epsilon^2$, one finds
\begin{equation}\label{eq:Fsecond-upper-equilibrium}
 F_\epsilon''(Q_\epsilon^*)
 =2s_\epsilon(2s_\epsilon-1)\longrightarrow\frac34.
\end{equation}
Hence this second derivative is bounded above and below by positive
constants for all sufficiently small $\epsilon$.

\begin{figure}[H]
\centering
\includegraphics[width=0.68\textwidth]{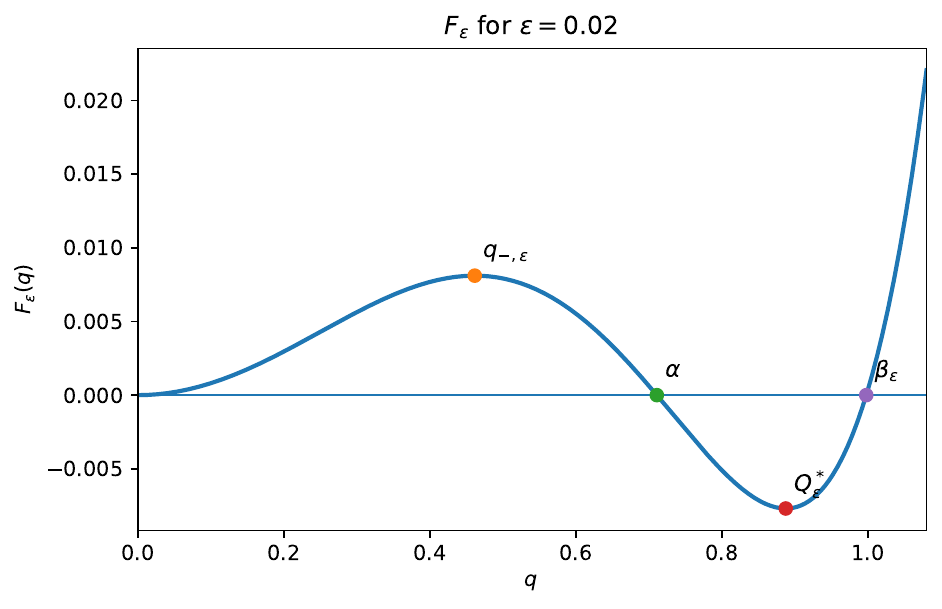}
\caption{The effective potential \(F_\epsilon\) for the representative value
\(\epsilon=0.02\).  The points \(q_{-,\epsilon}\) and \(Q_\epsilon^*\)
are the positive critical points of \(F_\epsilon\), while \(\alpha\) and
\(\beta_\epsilon\) are its positive nonzero roots.  In particular,
\(F_\epsilon<0\) on \((\alpha,\beta_\epsilon)\), and
\(Q_\epsilon^*\) is the upper local minimum relevant to the flat-top
ground state.  This schematic numerical plot is included  to
visualize the sign structure used below.}
\label{fig:F-epsilon}
\end{figure}

Figure~\ref{fig:F-epsilon} summarizes the geometry behind
\eqref{eq:F-factorization}--\eqref{eq:Fsecond-upper-equilibrium}: the
center value of the ground state lies to the right of the zero
\(\alpha\), but to the left of the upper equilibrium \(Q_\epsilon^*\).

Since $Q_\omega$ is radial, we identify $Q_\omega(x)=Q_\omega(r)$ with $r=|x|$. Then $Q_\omega$ solves
\begin{equation*}
    -Q_\omega''(r)-\frac{2}{r}Q_\omega'(r)+F'_\epsilon(Q_\omega)=0
\end{equation*}
with boundary conditions
$Q_\omega'(0)=0$ and $Q_\omega(+\infty)=Q_\omega'(+\infty)=0$.
Multiplying the equation by \(Q_\omega'\) and integrating, we obtain
\begin{align*}
    \partial_r\left[\frac{(Q_\omega')^2}{2}-F_\epsilon(Q_\omega)\right]=-\frac{2}{r}(Q_\omega')^2\le 0,
\end{align*}
hence
\begin{align*}
  \frac{(Q_\omega'(r))^2}{2}-F_\epsilon(Q_\omega(r)) = \int_r^{+\infty}  \frac{2}{s}(Q_\omega')^2 ds.
\end{align*}
In particular, 
$$F_\epsilon(Q_\omega(0)) = -\int_0^{+\infty}  \frac{2}{r}(Q_\omega')^2 dr < 0.$$ 

At the origin, radial regularity gives
$Q_\omega'(0)=0$ and
\[
 \Delta Q_\omega(0)=3Q_\omega''(0).
\]
Because $Q_\omega$ has its strict maximum at the origin, the stationary
equation implies
\[
 F_\epsilon'(Q_\omega(0))
 =\Delta Q_\omega(0)=3Q_\omega''(0)\leq0.
\]
The largest positive zero of $F_\epsilon'$ is $Q_\epsilon^*$, by
\eqref{eq:Fprime-factorization}; hence $Q_\omega(0)\leq Q_\epsilon^*$.
Equality would give the initial data
$Q_\omega(0)=Q_\epsilon^*$ and $Q_\omega'(0)=0$ at an equilibrium of the
radial ODE.  Uniqueness for the regular radial initial-value problem would
then force $Q_\omega\equiv Q_\epsilon^*$, contradicting
$Q_\omega(r)\to0$.  Therefore $Q_\omega(0)<Q_\epsilon^*$.

On the other hand, the strict energy identity above gives
$F_\epsilon(Q_\omega(0))<0$.  Since $Q_\omega(0)>0$, the factorization
\eqref{eq:F-factorization} yields $Q_\omega(0)>\alpha$.  We have proved
\begin{equation}\label{eq:apriori-center-bound}
 \alpha<Q_\omega(0)<Q_\epsilon^*.
\end{equation}
Notice also that $Q_\epsilon^*-\alpha=O(\sqrt\epsilon)$, so
$Q_\omega(0)$ lies in an $O(\sqrt\epsilon)$ neighborhood of the upper
equilibrium.

Next, define \(R_\epsilon>0\) by the condition
\begin{equation}\label{eq:Redef}
Q_\omega(R_\epsilon)=\sqrt{\frac{3}{8}}.
\end{equation} Since \(Q_\omega\) is strictly decreasing, this point is uniquely determined. As in \cite{killip2017solitons}, we have the a priori estimate \(R_\epsilon\approx\epsilon^{-1}\).

We shall prove later, in Proposition~\ref{asymptotic-R}, that
\[
R_\epsilon=\frac{\sqrt{3}}{4\epsilon}+\mathcal{O}(1).
\]

\begin{figure}[H]
\centering
\includegraphics[width=0.72\textwidth]{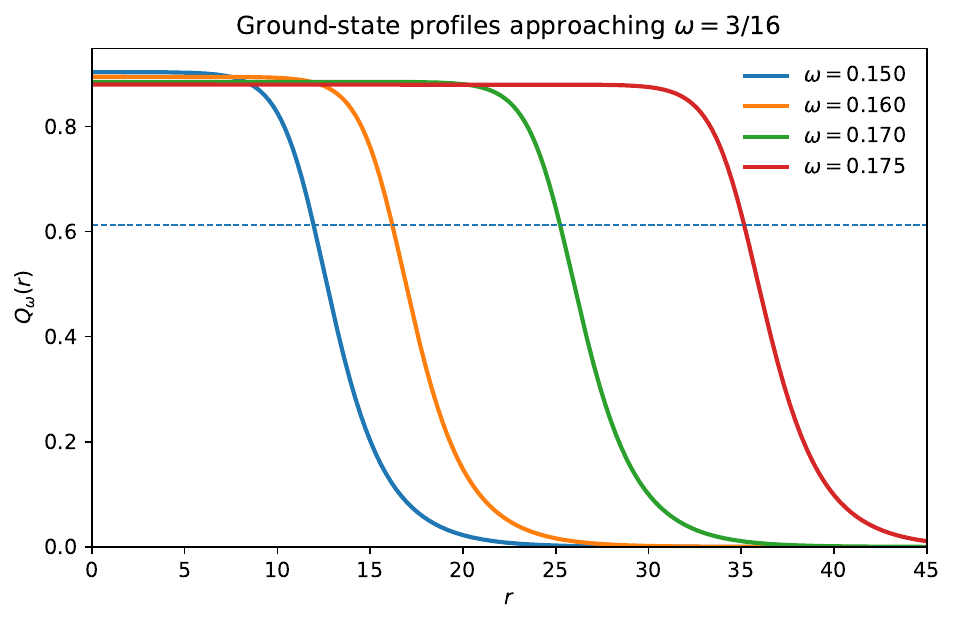}
\caption{Numerically computed positive radial ground states
\(Q_\omega\) for several values of \(\omega\) approaching \(3/16\).
The dashed horizontal line is the midpoint level \(\sqrt{3/8}\) used to
define \(R_\epsilon\).  The figure illustrates the flat-top regime:
the plateau height approaches \(Q_\epsilon^*\approx\sqrt{3/4}\), while
the transition radius increases like \(\epsilon^{-1}\)..}
\label{fig:Q-omega-family}
\end{figure}

The family in Figure~\ref{fig:Q-omega-family} illustrates 
three regions treated separately below: an almost constant interior
plateau, an \(O(1)\)-width interface near \(R_\epsilon\), and an
exponentially decaying exterior tail.

Inspired by the observation above, the strategy of this section is to combine local comparison arguments with explicit one-dimensional transition profiles. We first obtain exponential control near the plateau and in the far field, then compare the exact radial profile with the autonomous solutions $P_1$ and $P_2$ in the transition region, and finally identify the limiting heteroclinic profile $P_0$ and the asymptotic location of $R_\epsilon$.                                                                                                                                                                                                                                                                                                                                                                                                                                                                                                                                                                                                                                                                                                                                                                                                                                                                                                                                                                                                                                                                                                                                                                                                                                                                                                                                                                                                    

\subsection{Asymptotics of $Q_\omega$ when $r\to 0$}

Define \(y=r(Q_\epsilon^*-Q_\omega)\). Then \(y\) solves
\begin{align*}
    y''= -rF_\epsilon'(Q_\omega)=\frac{F_\epsilon'(Q_\epsilon^*)-F_\epsilon'(Q_\omega)}{Q_\epsilon^*-Q_\omega} y.
\end{align*}
Choose \(\theta(r)\in(Q_\omega(r),Q_\epsilon^*)\) so that
$$F''_\epsilon(\theta)=\frac{F_\epsilon'(Q_\epsilon^*)-F_\epsilon'(Q_\omega)}{Q_\epsilon^*-Q_\omega},
$$
Then
\begin{align*}
    y''= F''_\epsilon(\theta) y.
\end{align*}
We shall use the following elementary comparison lemma.
\begin{lem}[Comparison lemma]
    Suppose \(f\geq0\) and \(g\geq0\) satisfy
    \(f''=a(t)f\) and \(g''=b(t)g\), where \(a(t)\leq b(t)\), \(f(t_1)=g(t_1)=0\), and \(f(t_2)=g(t_2)>0\). Then \(f(t)\geq g(t)\) for \(t\in(t_1,t_2)\).
\end{lem}
\begin{proof}
    Define \(W=f'g-g'f\). Then \(W'=f''g-g''f=(a-b)fg\leq0\). Since \(W(t_1)=0\), we have \(W(t)\leq0\) for \(t\in[t_1,t_2]\). Moreover, \(\left(\frac{f}{g}\right)'=\frac{f'g-g'f}{g^2}\leq0\). Since \(\frac{f(t_2)}{g(t_2)}=1\), it follows that \(f(t)\geq g(t)\) for \(t\in(t_1,t_2)\).
\end{proof}

We make the uniformity in this local comparison explicit.  By
\eqref{eq:apriori-center-bound},
$Q_\epsilon^*-Q_\omega(0)=O(\sqrt\epsilon)$, while
\eqref{eq:Fsecond-upper-equilibrium} gives
$F_\epsilon''(Q_\epsilon^*)\to3/4$.  A short bootstrap based on
$y''=F_\epsilon''(\theta)y$, $y(0)=0$, and
$y'(0)=Q_\epsilon^*-Q_\omega(0)$ shows that, on some fixed interval
$[0,r_0]$, the values of $Q_\omega(r)$ remain in a fixed small
neighborhood of $Q_\epsilon^*$.  Continuity of $F_\epsilon''$ then gives
constants $0<\lambda_-<\lambda_+<\infty$, independent of sufficiently
small $\epsilon$, such that
\begin{equation}\label{eq:uniform-local-convexity}
 \lambda_-<F_\epsilon''(\theta(r))<\lambda_+,
 \qquad 0\leq r\leq r_0.
\end{equation}
Indeed, one may first choose a neighborhood on which
$F_\epsilon''\in[1/2,1]$ for all small $\epsilon$, and then choose $r_0$
by continuity and the preceding bootstrap.

Consider
\begin{align*}
    &y_+''=\lambda_+y, \quad  y_+(0)=0, \quad y_+(r_0)=y(r_0),\\
    &y_-''=\lambda_-y, \quad y_-(0)=0, \quad y_-(r_0)=y(r_0).
\end{align*}
By the comparison lemma, we have
\begin{align*}
    y_+(r)\le y(r) \le y_-(r).
\end{align*}
Hence we have
$$
\frac{e^{\sqrt{\lambda_+}r}-e^{-\sqrt{\lambda_+}r}}{e^{\sqrt{\lambda_+}r_0}-e^{-\sqrt{\lambda_+}r_0}}y(r_0)\le y(r)\le \frac{e^{\sqrt{\lambda_-}r}-e^{-\sqrt{\lambda_-}r}}{e^{\sqrt{\lambda_-}r_0}-e^{-\sqrt{\lambda_-}r_0}}y(r_0),
$$
and
\begin{align}\label{asymptotic-Q-near0}
Q_\epsilon^*-\frac{e^{\sqrt{\lambda_-}r}-e^{-\sqrt{\lambda_-}r}}{e^{\sqrt{\lambda_-}r_0}-e^{-\sqrt{\lambda_-}r_0}}\frac{r_0}{r}(Q_\epsilon^*-Q_\omega(r_0))\le Q_\omega(r)\le Q_\epsilon^*-\frac{e^{\sqrt{\lambda_+}r}-e^{-\sqrt{\lambda_+}r}}{e^{\sqrt{\lambda_+}r_0}-e^{-\sqrt{\lambda_+}r_0}}\frac{r_0}{r}(Q_\epsilon^*-Q_\omega(r_0)).
\end{align}

We next recover the derivative estimate from the profile estimate.  This
step uses the ODE and is not a formal differentiation of the preceding comparison estimate.
Since
\[
 y=r(Q_\epsilon^*-Q_\omega),
 \qquad
 ry'-y=-r^2Q_\omega',
\]
and $(ry'-y)'=ry''$, regularity at the origin gives the exact identity
\begin{equation}\label{eq:derivative-recovery-identity}
 -r^2Q_\omega'(r)
 =\int_0^r sF_\epsilon''(\theta(s))y(s)\,ds.
\end{equation}
The upper comparison for $y$ and \eqref{eq:uniform-local-convexity} imply,
with $a=\sqrt{\lambda_-}$,
\[
 |Q_\omega'(r)|
 \lesssim \frac{r_0}{r^2}e^{-a(r_0-r)}
       \int_0^r s e^{-a(r-s)}\,ds.
\]
For $0<r\le1$ the last integral is $O(r^2)$, whereas for $r\ge1$ it is
$O(r)$.  Thus
\begin{equation}\label{eq:local-Q-derivative}
 |Q_\omega'(r)|
 \lesssim \frac{r_0}{1+r}
 e^{-\sqrt{\lambda_-}(r_0-r)},
 \qquad 0<r\le r_0.
\end{equation}

The comparison argument above is local in the sense that $r_0$ is fixed independently of $\epsilon$.  It therefore gives exponential control of the profile on every fixed interval near the origin, but it does \emph{not} by itself justify taking $r_0$ of order $R_\epsilon$.  The global size of the plateau will be recovered later, after the transition analysis, by comparing $Q_\omega$ with the heteroclinic profile $P_2$ and then determining the asymptotics of $R_\epsilon$.  We will therefore refrain from using \eqref{asymptotic-Q-near0} with $r_0\sim R_\epsilon$ at this stage.

\subsection{Asymptotics of $Q_\omega$ when $r\to \infty$}
Define \(y=rQ_\omega\). Then \(y\) solves
\begin{align*}
    y''= \left[(\frac{3}{16}-\epsilon)-Q_\omega^2+Q_\omega^4\right]y,
\end{align*}
Choose \(\tilde r_0\) such that, for \(\tilde r_0\leq r<\infty\),
\(0<\tilde{\lambda}_-<\left[(\frac{3}{16}-\epsilon)-Q_\omega^2+Q_\omega^4\right]<\tilde{\lambda}_+\)
for some \(\tilde{\lambda}_-\) and \(\tilde{\lambda}_+\). Applying the same comparison lemma, we obtain
\begin{align}\label{asymptotic-Q-nearinfty}
    e^{-\sqrt{\tilde{\lambda}_+}(r-\tilde r_0)}\frac{\tilde r_0}{r}Q_\omega(\tilde r_0) \le Q_\omega(r)\le e^{-\sqrt{\tilde{\lambda}_-}(r-\tilde r_0)}\frac{\tilde r_0}{r}Q_\omega(\tilde r_0).
\end{align}
Similarly, we also have
\begin{align}\label{asymptotic-Q'-nearinfty}
    |Q_\omega'(r)|\lesssim \frac{\tilde r_0}{r}e^{-\sqrt{\tilde{\lambda}_-}(r-\tilde r_0)}.
\end{align}

\
\subsection{Asymptotics of $Q_\omega$ in the transition region}
In the transition region $r\approx R_\epsilon$, the first-order term $-\frac{2}{r}Q_\omega'(r)$ in the equation for $Q_\omega$ is small. It is therefore natural to compare $Q_\omega$ with solutions of the autonomous equation
\begin{align*}
    -P''+F_\epsilon'(P)=0.
\end{align*}
This equation has the conserved energy
\begin{align*}
    \frac{(P')^2}{2}-F_\epsilon(P)=C.
\end{align*}
Choosing $C=0$, we obtain the homoclinic solution 
\begin{equation*}
    P_1(x,x_0)=\sqrt{\frac{\frac{3}{4}-4\epsilon}{1+\sqrt{\frac{16}{3}\epsilon}\cosh\!\left(2\sqrt{\frac{3}{16}-\epsilon}\,(x-x_0)\right)}},
\end{equation*}
with any phase shift $x_0$. Note that $P_1(-\infty,x_0)=P_1(+\infty,x_0)=0$.

Choosing instead \(C=-F_\epsilon(Q_\epsilon^*)\), we obtain the heteroclinic solution
\begin{align*}
P_2(x,x_0)
\;=\;
-\frac{Q_\epsilon^*\,\sinh\!\big(\kappa(x-x_{0})\big)}
{\sqrt{\frac{(Q_\epsilon^*)^{2}}{2(Q_\epsilon^*)^{2}-\frac{3}{2}}+\,\cosh^{2}\!\big(\kappa(x-x_{0})\big)}},
\end{align*}
with
\(\kappa=\sqrt[4]{\frac{1}{16}+\epsilon}\sqrt{\frac{1}{2}+\sqrt{\frac{1}{16}+\epsilon}}\). Note that \(P_2(-\infty,x_0)=Q_\epsilon^*\) and \(P_2(+\infty,x_0)=-Q_\epsilon^*\).

Choose \(x_1\) and \(x_2\) so that
\(P_1(0,x_1)=P_2(0,x_2)=Q_\omega(R_\epsilon)=\sqrt{\frac{3}{8}}\) and \(P_1'(0,x_1)<0\). For simplicity, we continue to denote these profiles by \(P_1\) and \(P_2\).

\begin{figure}[H]
\centering
\includegraphics[width=0.72\textwidth]{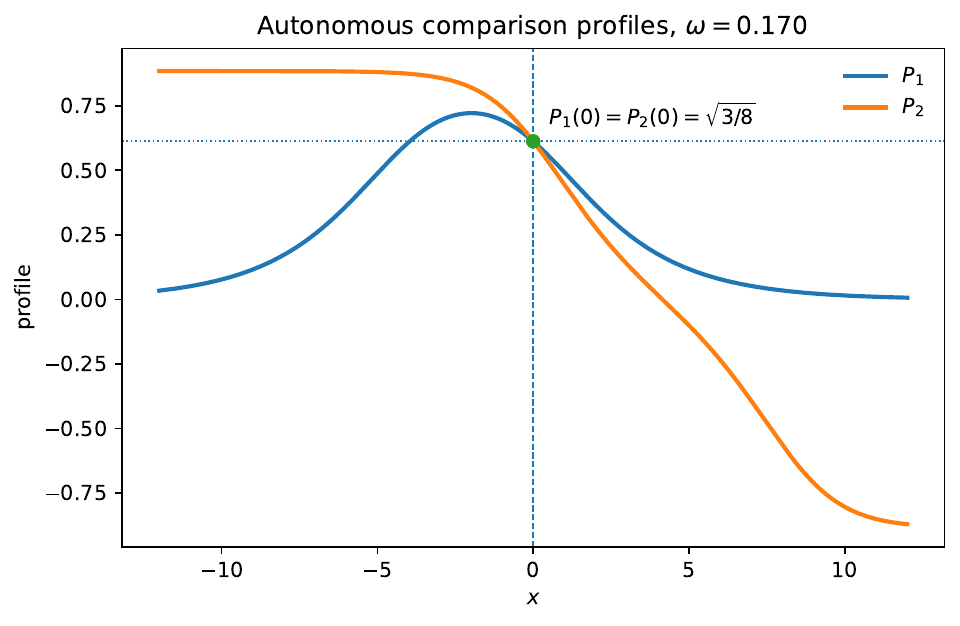}
\caption{ The two autonomous comparison profiles for the
representative value \(\omega=0.17\), translated so that
\(P_1(0)=P_2(0)=\sqrt{3/8}\) and \(P_1'(0)<0\).
The homoclinic \(P_1\) has zero conserved energy and returns to \(0\);
the heteroclinic \(P_2\) has conserved energy
\(-F_\epsilon(Q_\epsilon^*)\) and connects \(Q_\epsilon^*\) to
\(-Q_\epsilon^*\).  In the comparison with the positive radial ground
state, the relevant branches are \(P_2\) on the left of the interface
and \(P_1\) on the right.}
\label{fig:P1-P2}
\end{figure}

Figure~\ref{fig:P1-P2} also explains the choice of the glued profile
introduced below: \(P_2\) has the correct upper plateau as
\(x\to-\infty\), whereas \(P_1\) has the correct positive tail as
\(x\to+\infty\).

We now claim that 
\begin{prop}
The following inequalities hold:
$P_1(r-R_\epsilon)<Q_\omega(r)<P_2(r-R_\epsilon)$ for $0<r<R_\epsilon$, while $P_2(r-R_\epsilon)<Q_\omega(r)<P_1(r-R_\epsilon)$ for $R_\epsilon<r<+\infty$.
\end{prop}
\begin{proof}
We first compare the slopes at the matching point $r=R_\epsilon$.  The exact profile satisfies
\[
\frac12(Q_\omega'(r))^2-F_\epsilon(Q_\omega(r))=\int_r^{+\infty}\frac{2}{s}(Q_\omega'(s))^2\,ds,
\]
whereas the autonomous profiles satisfy
\[
\frac12(P_1')^2-F_\epsilon(P_1)=0,
\qquad
\frac12(P_2')^2-F_\epsilon(P_2)=-F_\epsilon(Q_\epsilon^*).
\]

Before evaluating at the matching value, note that the upper slope bound
uses the center estimate proved above.  Indeed,
\[
 0<\int_{R_\epsilon}^{\infty}\frac2s(Q_\omega')^2\,ds
 <\int_0^{\infty}\frac2s(Q_\omega')^2\,ds
 =-F_\epsilon(Q_\omega(0))
 <-F_\epsilon(Q_\epsilon^*).
\]
The last inequality follows because $Q_\epsilon^*$ is the strict local
minimum of $F_\epsilon$ in the interval containing $Q_\omega(0)$ and
$Q_\omega(0)<Q_\epsilon^*$.  Thus the exact radial dissipation lies
strictly between the conserved-energy constants of $P_1$ and $P_2$.

Evaluating these identities at the common value $Q_\omega(R_\epsilon)=P_1(0)=P_2(0)=\sqrt{3/8}$ yields
\[
|P_1'(0)|<|Q_\omega'(R_\epsilon)|<|P_2'(0)|.
\]
Since all three profiles are decreasing through the transition point, this can be rewritten as
\begin{equation}\label{eq:slope-order-R}
0>P_1'(0)>Q_\omega'(R_\epsilon)>P_2'(0).
\end{equation}

We prove only one inequality; the others are similar. Consider \(r<R_\epsilon\). At \(r=R_\epsilon\), we have \(Q_\omega(R_\epsilon)=P_1(0)\) and, by \eqref{eq:slope-order-R},
\[
Q_\omega'(R_\epsilon)<P_1'(0).
\]
Hence for $r<R_\epsilon$ but sufficiently close to $R_\epsilon$ one has $Q_\omega(r)>P_1(r-R_\epsilon)$.  Suppose, toward a contradiction, that this inequality fails somewhere on $(0,R_\epsilon)$.  Then there exists a first crossing point $r_0\in(0,R_\epsilon)$ such that
\[
Q_\omega(r)>P_1(r-R_\epsilon)\quad\text{for }r\in(r_0,R_\epsilon),
\qquad
Q_\omega(r_0)=P_1(r_0-R_\epsilon).
\]
By first-contact monotonicity, one must have
\[
Q_\omega'(r_0)\ge P_1'(r_0-R_\epsilon).
\]
On the other hand, evaluating the first integrals at $r=r_0$ and using the same identity as at $r=R_\epsilon$, we get
\[
|Q_\omega'(r_0)|^2=2F_\epsilon(Q_\omega(r_0))+2\int_{r_0}^{+\infty}\frac{2}{s}(Q_\omega'(s))^2\,ds
>2F_\epsilon(P_1(r_0-R_\epsilon))=|P_1'(r_0-R_\epsilon)|^2.
\]
Since both derivatives are negative, this implies
\[
Q_\omega'(r_0)<P_1'(r_0-R_\epsilon),
\]
which contradicts the first-contact inequality.  Thus
\[
P_1(r-R_\epsilon)<Q_\omega(r),\qquad 0<r<R_\epsilon
\]as desired.
\end{proof}

We next quantify the error in this comparison.

\begin{prop}\label{error-estimates-P1P2}
The following inequalities hold: \(|P_2(r-R_\epsilon)-Q_\omega(r)|\lesssim\epsilon\) and \(|P_2'(r-R_\epsilon)-Q_\omega'(r)|\lesssim\epsilon\) for \(0\leq r\leq R_\epsilon\); and \(|P_1(r-R_\epsilon)-Q_\omega(r)|\lesssim\epsilon\) and \(|P_1'(r-R_\epsilon)-Q_\omega'(r)|\lesssim\epsilon\) for \(r\geq R_\epsilon\).
\end{prop}
\begin{proof}

We give the argument on $0\le r\le R_\epsilon$ in detail; the exterior
argument is identical after replacing $P_2$ by $P_1$ and using the
far-field estimates.  Write
\[
 P(r):=P_2(r-R_\epsilon),
 \qquad W(r):=P(r)-Q_\omega(r)\ge0.
\]
Then
\begin{equation}\label{eq:W-transition-equation}
 -W''+\frac2rQ_\omega'
 +F_\epsilon'(P)-F_\epsilon'(Q_\omega)=0.
\end{equation}
At $r=R_\epsilon$ one has $W=0$.  At $r=0$, the plateau estimates for
$Q_\omega$ and the explicit left tail of $P_2$ give
$W(0)=O(e^{-cR_\epsilon})$.

Let $r_m\in(0,R_\epsilon)$ be an interior point where $W$ attains its
maximum.  If $r_m\le1$, the same plateau estimates already give
$W(r_m)=O(e^{-cR_\epsilon})$.  Assume henceforth that $r_m>1$.  Since
$W'(r_m)=0$, we have $P'(r_m)=Q_\omega'(r_m)$; since $W''(r_m)\le0$,
\eqref{eq:W-transition-equation} yields
\begin{equation}\label{eq:W-max-inequality}
 c_m W(r_m)\le -\frac2{r_m}P'(r_m),
 \qquad
 c_m:=\frac{F_\epsilon'(P(r_m))-F_\epsilon'(Q_\omega(r_m))}
 {P(r_m)-Q_\omega(r_m)}.
\end{equation}
Choose a fixed small $\delta_0>0$ on which
$F_\epsilon''\ge c_0>0$ uniformly near $Q_\epsilon^*$.  If
$Q_\epsilon^*-Q_\omega(r_m)\le\delta_0$, then both values in the quotient
lie in this convexity interval, so $c_m\ge c_0$.  The explicit heteroclinic
tail gives
$|P'(r_m)|\lesssim e^{-c(R_\epsilon-r_m)}$, and the elementary bound
\[
 \sup_{1\le r\le R_\epsilon}
 \frac{e^{-c(R_\epsilon-r)}}r\lesssim R_\epsilon^{-1}
\]
turns \eqref{eq:W-max-inequality} into
$W(r_m)\lesssim R_\epsilon^{-1}\lesssim\epsilon$.

If instead $Q_\epsilon^*-Q_\omega(r_m)>\delta_0$, the global plateau
estimate implies $|r_m-R_\epsilon|\le C_{\delta_0}$.  Proposition~2.2
then sandwiches $Q_\omega$ between $P_1$ and $P_2$, while the explicit
formulas (equivalently, Proposition~2.4 below) give
\[
 \sup_{|x|\le C_{\delta_0}}|P_2(x)-P_1(x)|\lesssim\epsilon.
\]
Hence $W(r_m)\lesssim\epsilon$ in this case as well.  This proves
\begin{equation}\label{eq:prop23-position-left}
 \sup_{0\le r\le R_\epsilon}|P_2(r-R_\epsilon)-Q_\omega(r)|
 \lesssim\epsilon.
\end{equation}

We next estimate the derivatives.  Set
$U=Q_\omega'-P'=-W'$.  It is important that this estimate is obtained from
the differential equations, not by differentiating
\eqref{eq:prop23-position-left}.  In the part of the interface where
$Q_\epsilon^*-Q_\omega\ge\delta_0$, all relevant slopes are bounded away
from zero.  The first integrals give
\[
 Q_\omega'=-\sqrt{2F_\epsilon(Q_\omega)+2D_\omega(r)},
 \qquad
 P'=-\sqrt{2F_\epsilon(P)-2F_\epsilon(Q_\epsilon^*)},
\]
where
$D_\omega(r)=\int_r^\infty (2/s)(Q_\omega')^2\,ds$.
On every fixed-width interface strip,
$D_\omega(r)=O(R_\epsilon^{-1})=O(\epsilon)$ and
$F_\epsilon(Q_\epsilon^*)=O(\epsilon)$.  Together with
\eqref{eq:prop23-position-left}, the Lipschitz continuity of the square
root away from zero gives $|U|\lesssim\epsilon$ there.

It remains to treat the convex plateau part.  Differentiating
\eqref{eq:W-transition-equation} gives
\begin{equation}\label{eq:U-transition-equation}
 -U''+F_\epsilon''(Q_\omega)U
 =-\frac2{r^2}Q_\omega'+\frac2rQ_\omega''
 +\bigl(F_\epsilon''(P)-F_\epsilon''(Q_\omega)\bigr)P'.
\end{equation}
The already established position estimate, the plateau estimates for
$Q_\omega,Q_\omega'$, and the explicit bounds for $P,P'$ show that the
right-hand side is $O(\epsilon)$ uniformly for $r\ge1$.  At an interior
maximum point of $|U|$, one has $U''U\le0$; since
$F_\epsilon''(Q_\omega)\ge c_0$, equation
\eqref{eq:U-transition-equation} gives $|U|\lesssim\epsilon$.  The
endpoints and $0\le r\le1$ are covered by the exponential plateau bounds.
Consequently,
\[
 \sup_{0\le r\le R_\epsilon}
 |P_2'(r-R_\epsilon)-Q_\omega'(r)|\lesssim\epsilon.
\]

For $r\ge R_\epsilon$, use $P=P_1(r-R_\epsilon)$ and
$W=Q_\omega-P_1\ge0$.  Near the interface the same first-integral
argument applies, now with conserved energy zero for $P_1$; in the small
amplitude region, $F_\epsilon''(0)=\omega$ is uniformly positive and the
analogue of \eqref{eq:U-transition-equation} applies.  The far-field
exponential bounds handle the endpoint at infinity.  This yields
\[
 |P_1(r-R_\epsilon)-Q_\omega(r)|
 +|P_1'(r-R_\epsilon)-Q_\omega'(r)|\lesssim\epsilon,
 \qquad r\ge R_\epsilon,
\]
and completes the proof.
\end{proof}

\subsection{Approximate solutions}
We construct a family of approximate solutions as follows. Define
\begin{align*}
    P_\epsilon(x)=\left\{
    \begin{aligned}
        & P_2(x), x\le 0,\\
        & P_1(x), x\ge 0,\\
        & P_1(0)=P_2(0)=\sqrt{\frac{3}{8}},
    \end{aligned} 
    \right.
\end{align*}
Then $P_\epsilon(x-R_\epsilon)$ is the approximate profile that models the exact radial solution in the transition region. In view of \eqref{asymptotic-Q-near0}, \eqref{eq:local-Q-derivative}, \eqref{asymptotic-Q-nearinfty}, \eqref{asymptotic-Q'-nearinfty}, and Proposition~\ref{error-estimates-P1P2}, we have
\begin{align}\label{error-estimates-Pepsilon}
    |Q_\omega(r)-P_\epsilon(r-R_\epsilon)|\lesssim \min\{\epsilon, e^{-c|r-R_\epsilon|}\},
\end{align}
and
\begin{align}\label{error-estimates-P'epsilon}
    |Q_\omega'(r)-P_\epsilon'(r-R_\epsilon)|\lesssim \min\{\epsilon, e^{-c|r-R_\epsilon|}\}.
\end{align}

\begin{figure}[H]
\centering
\includegraphics[width=0.72\textwidth]{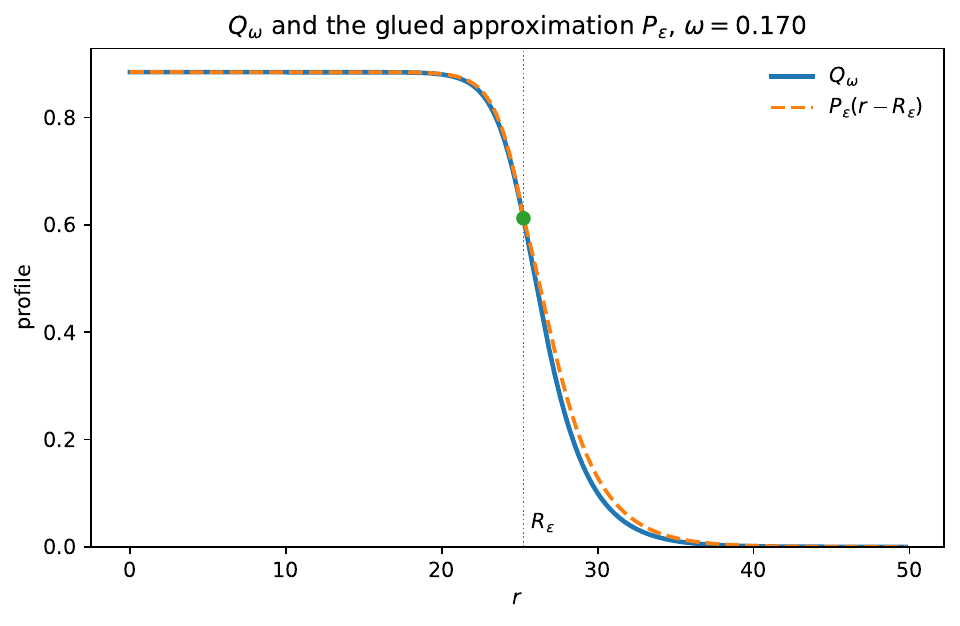}
\caption{ A numerical ground state \(Q_\omega\) and the glued
autonomous approximation \(P_\epsilon(r-R_\epsilon)\) for
\(\omega=0.17\).  Here \(R_\epsilon\) is determined from the numerical
ground state by \(Q_\omega(R_\epsilon)=\sqrt{3/8}\), and
\(P_\epsilon=P_2\) to the left of the interface and \(P_\epsilon=P_1\)
to the right.  The curves are visually almost indistinguishable on the
scale of the plot, consistent with
\eqref{error-estimates-Pepsilon}--\eqref{error-estimates-P'epsilon}.}
\label{fig:Q-vs-Pepsilon}
\end{figure}

Figure~\ref{fig:Q-vs-Pepsilon} shows geometrically what
Proposition~\ref{error-estimates-P1P2} proves analytically: the curvature
term \(2Q_\omega'/r\) produces only a small perturbation of the
one-dimensional autonomous interface once the transition is centered at
\(R_\epsilon\).

Define
\begin{align*}
    P_0(x)
=\sqrt{\frac{3}{8}\left(1-\tanh\!\left(\frac{\sqrt3}{4}x\right)\right)},
\end{align*} 
which is (up to translation) the unique heteroclinic solution satisfying
\begin{align*}
    -P_0''+F_0'(P_0)=0.
\end{align*}
\begin{rem}
After rescaling, \(P_0\) is exactly the kink connecting the two vacua \(m\) and \(0\) in the \(\phi^6\) model in \(P(\phi)_2\) theory:
\begin{equation*}
    \phi''=W_6'(\phi), \quad W_6(\phi)=\phi^2(\phi^2-m^2)^2.
\end{equation*}
\end{rem}
The next proposition shows that $P_0$ is the limiting profile of $P_\epsilon$ as $\epsilon\to0$.

\begin{prop}\label{error-estimates-P0}
We have
$$
\|P_\epsilon-P_0\|_{L^\infty}\lesssim\epsilon,
$$
and
$$
\|P'_\epsilon-P'_0\|_{L^\infty}\lesssim \epsilon.
$$
\end{prop}
\begin{proof}
The estimates follow directly by differentiating the explicit formulas
for \(P_1\) and \(P_2\) on the two half-lines and applying the mean-value
theorem in \(\epsilon\).  The midpoint normalization is the same for all
three profiles, and the resulting bounds are uniform up to the one-sided
derivatives at \(x=0\).\end{proof}

Combining \eqref{error-estimates-Pepsilon}, \eqref{error-estimates-P'epsilon}, and Proposition~\ref{error-estimates-P0}, we obtain
\begin{cor}\label{cor:P0-approx}
We have
\begin{align*}
    \|Q_\omega(r)-P_0(r-R_\epsilon)\|_{L^\infty}\lesssim\epsilon,
\end{align*}
and
\begin{align*}
    \|Q_\omega'(r)-P_0'(r-R_\epsilon)\|_{L^\infty}\lesssim \epsilon.
\end{align*}
\end{cor}

\subsection{Asymptotics of the phase translation $R_\epsilon$}
In this subsection, we derive a sharper asymptotic formula for \(R_\epsilon\).

\begin{prop}\label{asymptotic-R}Given $R_\epsilon$ defined as \eqref{eq:Redef}, we have
    $$R_\epsilon=\frac{\sqrt{3}}{4\epsilon}+\mathcal{O}(1).$$
\end{prop}
\begin{proof}
Recall the exact identity
\begin{equation}\label{formula-FQ0}
F_\epsilon(Q_\omega(0))
=
-\int_0^{+\infty}\frac{2}{r}(Q_\omega'(r))^2\,dr.
\end{equation}
We compare the weighted dissipation integral on the right-hand side with the same quantity computed on the limiting interface $P_0$.  The proof is organized so that every approximation error is of order $\mathcal O(\epsilon^2)$.

Fix $A>0$ large and set
\[
L_\epsilon:=A\log \frac1\epsilon .
\]
We split the half-line into three regions:
\[
[0,R_\epsilon-L_\epsilon],\qquad [R_\epsilon-L_\epsilon,R_\epsilon+L_\epsilon],\qquad [R_\epsilon+L_\epsilon,+\infty).
\]
The middle interval is the transition zone; the two outer intervals are tails.

\smallskip
\noindent
We begin with the right tail. Since
\[
|Q_\omega'(r)|\le Ce^{-c(r-R_\epsilon)},
\qquad r\ge R_\epsilon,
\]
we have
\begin{align*}
\int_{R_\epsilon+L_\epsilon}^{+\infty}\frac{2}{r}(Q_\omega'(r))^2\,dr
&\le \frac{2}{R_\epsilon+L_\epsilon}\int_{R_\epsilon+L_\epsilon}^{+\infty}(Q_\omega'(r))^2\,dr\\
&\lesssim \frac1{R_\epsilon}\int_{L_\epsilon}^{+\infty}e^{-2cs}\,ds
\lesssim \frac{e^{-2cL_\epsilon}}{R_\epsilon}
\lesssim \epsilon^2,
\end{align*}
provided $A$ is chosen sufficiently large.

The left tail requires a slightly more careful split because of the weight $1/r$.  On the interval $[1,R_\epsilon-L_\epsilon]$, Proposition~\ref{error-estimates-P1P2} gives
\[
|Q_\omega'(r)|\le Ce^{-c(R_\epsilon-r)},
\]
and therefore
\begin{align}
\int_1^{R_\epsilon-L_\epsilon}\frac{2}{r}(Q_\omega'(r))^2\,dr
&\lesssim \int_1^{R_\epsilon-L_\epsilon} e^{-2c(R_\epsilon-r)}\,dr
\lesssim e^{-2cL_\epsilon}
\lesssim \epsilon^2 .\label{eq:left-tail-middle-final}
\end{align}
Near the origin, using the fact that
\[
Q_\omega'(r)=\frac{1}{r^2}\int_0^r s^2F_\epsilon'(Q_\omega(s)) ds,
\]
we have
\[
|Q_\omega'(r)|\lesssim r\,e^{-cR_\epsilon},
\qquad 0\le r\le1,
\]
and hence
\begin{align}
\int_0^1\frac{2}{r}(Q_\omega'(r))^2\,dr
&\lesssim e^{-2cR_\epsilon}\int_0^1 r\,dr
\lesssim e^{-2cR_\epsilon}
\lesssim \epsilon^2 .\label{eq:left-tail-origin-final}
\end{align}
Combining \eqref{eq:left-tail-middle-final} and \eqref{eq:left-tail-origin-final}, we obtain
\begin{equation*}
\int_0^{R_\epsilon-L_\epsilon}\frac{2}{r}(Q_\omega'(r))^2\,dr\lesssim \epsilon^2 .
\end{equation*}

Exactly the same exponential localization gives
\begin{equation}\label{eq:P0-tail-small-final}
\int_{|s|\ge L_\epsilon}(P_0'(s))^2\,ds\lesssim \epsilon^2 .
\end{equation}

On the interval $|r-R_\epsilon|\le L_\epsilon$, one has $r\sim R_\epsilon$, so
\[
\left|\frac2r-\frac2{R_\epsilon}\right|
=
\frac{2|r-R_\epsilon|}{rR_\epsilon}
\lesssim \frac{|r-R_\epsilon|}{R_\epsilon^2}.
\]
Hence
\begin{align*}
\int_{R_\epsilon-L_\epsilon}^{R_\epsilon+L_\epsilon}
\left|\frac2r-\frac2{R_\epsilon}\right|(Q_\omega'(r))^2\,dr
&\lesssim \frac1{R_\epsilon^2}
\int_{R_\epsilon-L_\epsilon}^{R_\epsilon+L_\epsilon}|r-R_\epsilon|(Q_\omega'(r))^2\,dr\\
&\lesssim \frac1{R_\epsilon^2}\int_{R_\epsilon-L_\epsilon}^{R_\epsilon+L_\epsilon}|r-R_\epsilon|e^{-c|r-R_\epsilon|} dr\\
&\lesssim \epsilon^2.
\end{align*}
Therefore
\begin{equation*}
\int_{R_\epsilon-L_\epsilon}^{R_\epsilon+L_\epsilon}\frac{2}{r}(Q_\omega'(r))^2\,dr
=
\frac{2}{R_\epsilon}
\int_{R_\epsilon-L_\epsilon}^{R_\epsilon+L_\epsilon}(Q_\omega'(r))^2\,dr
+\mathcal O(\epsilon^2).
\end{equation*}

We now replace $Q_\omega'$ by the limiting profile $P_0'(\cdot-R_\epsilon)$.
Corollary~\ref{cor:P0-approx} gives
\[
\bigl|(Q_\omega'(r))^2-(P_0'(r-R_\epsilon))^2\bigr|
\lesssim \epsilon e^{-c|r-R_\epsilon|}.
\]
Hence
\begin{align*}
\frac2{R_\epsilon}
\int_{R_\epsilon-L_\epsilon}^{R_\epsilon+L_\epsilon}
\bigl|(Q_\omega'(r))^2-(P_0'(r-R_\epsilon))^2\bigr|\,dr
&\lesssim \frac{\epsilon}{R_\epsilon}\int_{\mathbb R}e^{-c|s|}\,ds\\
&\lesssim \epsilon^2.
\end{align*}
Using \eqref{eq:P0-tail-small-final}, we therefore obtain
\begin{equation}\label{eq:dissipation-compare-final}
\left|
\int_0^\infty \frac{2}{r}(Q_\omega'(r))^2\,dr
-
\frac{2}{R_\epsilon}\int_{-\infty}^{+\infty}(P_0'(s))^2\,ds
\right|
\lesssim \epsilon^2 .
\end{equation}

The remaining integral is explicit.  From
\[
P_0(s)=\sqrt{\frac38\Bigl(1-\tanh\!\bigl(\tfrac{\sqrt3}{4}s\bigr)\Bigr)}
\]
one computes directly that
\[
\int_{-\infty}^{+\infty}(P_0'(s))^2\,ds=\frac{3\sqrt3}{64}.
\]
Substituting this into \eqref{eq:dissipation-compare-final} yields
\begin{equation*}
-\int_0^\infty \frac{2}{r}(Q_\omega'(r))^2\,dr
=
-\frac{3\sqrt3}{32R_\epsilon}+\mathcal O(\epsilon^2).
\end{equation*}
By \eqref{formula-FQ0}, this is equivalent to
\begin{equation}\label{eq:FQ0-vs-R-final-2}
F_\epsilon(Q_\omega(0))
=
-\frac{3\sqrt3}{32R_\epsilon}+\mathcal O(\epsilon^2).
\end{equation}

Since $Q_\omega(0)$ is exponentially close to $Q_\epsilon^*$, we have
\begin{equation}\label{eq:FQ0-expansion-final}
F_\epsilon(Q_\omega(0))
=F_\epsilon(Q_\epsilon^*)+\mathcal{O}(e^{-cR_\epsilon})
=
-\frac38\epsilon+\mathcal{O}(\epsilon^2).
\end{equation}
Comparing \eqref{eq:FQ0-vs-R-final-2} and \eqref{eq:FQ0-expansion-final}, we obtain
\[
\frac38\epsilon+\mathcal O(\epsilon^2)
=
\frac{3\sqrt3}{32R_\epsilon},
\]
thus we conclude
\[
R_\epsilon=\frac{\sqrt3}{4\epsilon}+\mathcal O(1)
\]
as claimed.
\end{proof}

\section{Discrete Spectrum of $L_+$}

In this section, we study the eigenvalues of $L^+$ below the bottom  $\omega$ of its continuous
spectrum.  Under the scalar Klein--Gordon evolution
\eqref{eq:scalar-kg-linearization}, positive such eigenvalues are localized
oscillatory modes, the unique negative radial eigenvalue is the unstable
scalar direction, and the zero eigenvalue in the $\ell=1$ sector is the
translation mode.

\subsection{Spherical harmonic decomposition of $L_+^{(\ell)}$}
Since $Q_\omega$ is radial, we may analyze the spectrum of $L_+$ by means of spherical harmonic decomposition. Given $f\in L^2(\mathbb{R}^3)$, let $r=|x|$ and $\theta=x/|x|$ denote the radial and angular variables, respectively. Then
\begin{align*}
f(r,\theta)=\sum_{\ell=0}^\infty\sum_{m=-\ell}^{\ell} f_{\ell m}(r)Y_{\ell m}(\theta),
\quad
\text{with}~~~~f_{\ell m}(r)=\int_{\mathbb S^{2}} f(r,\theta)\overline{Y_{\ell m}(\theta)}d\theta,
\end{align*}
where the spherical harmonics $Y_{\ell m}\in L^2(\mathbb{S}^2)$ are the eigenfunctions of the Laplace--Beltrami operator on the sphere:
$$
-\Delta_{\mathbb S^{2}}Y_{\ell m}=\ell(\ell+1)Y_{\ell m},\qquad \ell=0,1,2,\dots,\quad m=-\ell,\dots,\ell,
$$
which form an orthonormal basis of $L^{2}(\mathbb S^{2})$.
In particular, the spherical harmonic decomposition satisfies the Parseval identity:
$$
\|f\|_{L^{2}(\mathbb R^{3})}^{2}
=\sum_{\ell,m}\int_{0}^{\infty}|f_{\ell m}(r)|^{2} r^{2}dr.
$$
Under this decomposition, the linearized operator \(L_+\) decomposes as
$$L_+=\bigoplus_{\ell=0}^{\infty}L_+^{(\ell)},$$
with
\begin{align*}
L_+^{(\ell)}= -\partial_r^2-\frac{2}{r}\partial_r+ \omega - 3Q_\omega^2 + 5Q_\omega^4 + \frac{\ell(\ell+1)}{r^{2}}
\quad\text{on }L^2(0,+\infty;r^2\,dr),
\end{align*}
in the sense that 
\begin{align*}
    L_+f=\sum_{\ell=0}^\infty\sum_{m=-\ell}^{\ell} (L_+^{(\ell)}f_{\ell m})(r)Y_{\ell m}(\theta).
\end{align*}

\subsection{Eigenvalues of $L_+^{(\ell)}$}
Let 
\begin{align*}
    H^{(\ell)}=-\partial_r^2+ \omega - 3Q_\omega^2 + 5Q_\omega^4 + \frac{\ell(\ell+1)}{r^{2}}.
\end{align*}
We have \(rL_+^{(\ell)}f=H^{(\ell)}(rf)\). Moreover, \(H^{(\ell)}\) is self-adjoint on \(L^2(0,\infty;dr)\) as the Friedrichs realization with the regular endpoint condition at \(r=0\), namely
\begin{align*}
  \operatorname{Dom}(H^{(\ell)})
    =\left\{u\in H_0^1(0,\infty):
    -u''+\frac{\ell(\ell+1)}{r^2}u\in L^2(0,\infty)\right\}.
\end{align*}
The spectra of \(L_+^{(\ell)}\) and \(H^{(\ell)}\) are therefore identical. We first prove the following spectral property of \(L_+^{(\ell)}\):
\begin{prop}
For each $\ell\ge 0$, $L_+^{(\ell)}$ has at most one eigenvalue in $(-\infty, \omega)$.
\end{prop}
\begin{proof}

We separate the proof into two parts.  In the radial sector we shift
the interface to the origin and compare the half-line operator with the
one-dimensional kink operator on the whole line.  The latter has exactly
one negative eigenvalue and no zero-energy resonance.  Jost-function
convergence then transfers this count to the radial operator for small
$\epsilon$.  Once the radial count is known, the higher angular sectors
follow from a simple min--max induction.

We only need to prove this for $H^{(\ell)}$. We work with the shifted one-dimensional operators
\[
\widetilde H_\epsilon^{(\ell)}:=H^{(\ell)}-\omega
=-\partial_r^2-3Q_\omega^2+5Q_\omega^4+\frac{\ell(\ell+1)}{r^2}
\]
on $L^2(0,\infty;dr)$.  An eigenvalue of $L_+^{(\ell)}$ below $\omega$ is exactly a negative eigenvalue of $\widetilde H_\epsilon^{(\ell)}$, so it suffices to show that each $\widetilde H_\epsilon^{(\ell)}$ has at most one negative eigenvalue.

\smallskip
\noindent
\emph{Step 1: the radial sector.}
We begin with \(\ell=0\). Let \(H_{\epsilon}=-\partial_x^2-3P_\epsilon^2(x)+5P_\epsilon^4(x)\) denote the self-adjoint Schr\"odinger operator on \(L^2(\mathbb R)\). Set
    \begin{align*}
        V(x, \epsilon)=-3P_\epsilon^2(x)+5P_\epsilon^4(x),
    \end{align*}
    Then
    $$V_{+,\epsilon}:=\lim_{x\to +\infty}V(x,\epsilon)=0, \quad \text{and} \quad V_{-,\epsilon}:=\lim_{x\to -\infty}V(x,\epsilon)=-3(Q_\epsilon^*)^2+5(Q_\epsilon^*)^4>0.$$
   At \(\epsilon=0\), put
    \(\mathcal H_+=H_0+\omega_c\), \(\omega_c=3/16\), and
    \[
    A_+=\partial_x+\frac{\sqrt3}{8}
    \left(3\tanh\!\left(\frac{\sqrt3 x}{4}\right)-1\right).
    \]
    Direct calculation gives
    \[
    \mathcal H_+=A_+^*A_+,
    \qquad
    A_+A_+^*=-\partial_x^2+\omega_c
    +\frac9{64}\left(1-\tanh\!\left(\frac{\sqrt3 x}{4}\right)\right)^2.
    \]
    Hence \(0\) is the only eigenvalue of \(\mathcal H_+\) below
    \(\omega_c\), with eigenfunction \(P_0'\).  To exclude a threshold
    resonance, suppose \((\mathcal H_+-\omega_c)u=0\) with \(u\) bounded
    at both ends and put \(w=A_+u\).  Then
    \[
      -w''+\frac9{64}\left(1-\tanh\!\left(\frac{\sqrt3 x}{4}\right)\right)^2w=0.
    \]
    Multiplication by \(w\) and integration over \(\mathbb R\) give
    \(w=0\); hence \(u\in\ker A_+=\operatorname{span}\{P_0'\}\), which is
    incompatible with \(\mathcal H_+u=\omega_cu\).  Thus \(H_0\) has
    exactly one negative eigenvalue and neither an eigenvalue nor a
    resonance at zero.

    For any $z\in \mathbb{C}\setminus [0,+\infty)$, define the Jost functions $f_{\pm}(x,z,\epsilon)$ by
    \begin{align*}
       &H_{\epsilon} f_+(x,z,\epsilon)=zf_+(x,z,\epsilon), \quad f_+(x,z,\epsilon)\to e^{ixk_{+}} \quad\text{as } x\to +\infty;\\
       &H_{\epsilon} f_-(x,z,\epsilon)=zf_-(x,z,\epsilon), \quad f_-(x,z,\epsilon)\to e^{-ixk_{-,\epsilon}} \quad\text{as } x\to -\infty,
    \end{align*}
    where the branches \(k_+:=\sqrt{z}\) and \(k_{-,\epsilon}:=\sqrt{z-V_{-,\epsilon}}\) are chosen to have nonnegative imaginary parts. Define the Wronskian
    \begin{align*}
        W(z, \epsilon):=W[f_+(\cdot,z,\epsilon),f_-(\cdot,z,\epsilon)]:=f_+'(\cdot,z,\epsilon)f_-(\cdot,z,\epsilon)-f_-'(\cdot,z,\epsilon)f_+(\cdot,z,\epsilon).
    \end{align*}
    We first prove that \(W(z,\epsilon)\) tends to \(W(z,0)\) uniformly on every compact subset \(K\subset\mathbb{C}\setminus[0,+\infty)\). In fact, \(f_\pm(0,z,\epsilon)\) and \(f_\pm'(0,z,\epsilon)\) tend to \(f_\pm(0,z,0)\) and \(f_\pm'(0,z,0)\), respectively, uniformly on \(K\). Define
    \begin{align*}
        &m_+(x,z,\epsilon)=e^{-ixk_+}f_+(x,z,\epsilon),\\
        &m_-(x,z,\epsilon)=e^{ixk_{-,\epsilon}}f_-(x,z,\epsilon),
    \end{align*}
    and
    \begin{align*}
        &q_+(x,\epsilon)=V(x,\epsilon),\\
        &q_-(x,\epsilon)
    =V(x,\epsilon)-V_{-,\epsilon},
    \end{align*}
    Then
$$|q_+(x,\epsilon)-q_+(x,0)|\lesssim |P_\epsilon(x)-P_0(x)|\lesssim \min\{\epsilon, e^{-c|x|}\}.$$
 We have
 \begin{align*}
    &m_+(x,z,\epsilon)=1+\int_{x}^{\infty}\frac{e^{2ik_+(y-x)}-1}{2ik_+}\,q_{+}(y,\epsilon)\,m_+(y,z,\epsilon)\,dy,\\
    &m_-(x,z,\epsilon)=1+\int_{-\infty}^{x}\frac{e^{2ik_{-,\epsilon}(x-y)}-1}{2ik_{-,\epsilon}}\,q_{-}(y,\epsilon)\,m_-(y,z,\epsilon)\,dy.
 \end{align*}
 Since we have
 \begin{align*}
   \int_{0}^{\infty} y |q_+(y,\epsilon)|\,dy \lesssim1 \quad \text{and} \quad \int^{0}_{-\infty} |y| |q_-(y,\epsilon)|\,dy \lesssim 1,
 \end{align*}
 Gronwall's inequality gives
 \begin{align*}
     |m_+(x,z,\epsilon)|\lesssim 1 \quad \text{for} ~ x\ge 0, \quad \text{and} \quad |m_-(x,z,\epsilon)|\lesssim 1 \quad \text{for} ~ x\le 0.
 \end{align*}
Thus, when $x\leq0$,
\begin{align*}
    &|m_-(x,z,\epsilon)-m_-(x,z,0)|\\
    \lesssim&
    \int_{-\infty}^x \left| \frac{e^{2ik_{-,\epsilon}(x-y)}-1}{2ik_{-,\epsilon}}\,q_{-}(y,\epsilon)\,m_-(y,z,\epsilon)-\frac{e^{2ik_{-,0}(x-y)}-1}{2ik_{-,0}}\,q_{-}(y,0)\,m_-(y,z,0) \right|dy\\
    \lesssim& \int_{-\infty}^x \left| \frac{e^{2ik_{-,0}(x-y)}-1}{2ik_{-,0}}\,q_{-}(y,0)\,(m_-(y,z,\epsilon) -m_-(y,z,0) )\right|dy \\
    &+\int_{-\infty}^x \left| \left(\frac{e^{2ik_{-,\epsilon}(x-y)}-1}{2ik_{-,\epsilon}}-\frac{e^{2ik_{-,0}(x-y)}-1}{2ik_{-,0}} \right) \,q_{-}(y,0)\,m_-(y,z,\epsilon)\right|dy\\
    &+\int_{-\infty}^x \left| \frac{e^{2ik_{-,0}(x-y)}-1}{2ik_{-,0}}\,(q_{-}(y,\epsilon)-q_{-}(y,0))\,m_-(y,z,\epsilon) \right|dy\\
    :=& I + II + III.
\end{align*}
    We have $|\im k_{-,\epsilon}|\gtrsim 1$ for $z\in K$, so
    \begin{align*}
        I&= \int_{-\infty}^x \left| \frac{e^{2ik_{-,0}(x-y)}-1}{2ik_{-,0}}\,q_{-}(y,0)\,(m_-(y,z,\epsilon) -m_-(y,z,0) )\right|dy\\
        &\lesssim \int_{-\infty}^x | q_{-}(y,0)|\,|m_-(y,z,\epsilon) -m_-(y,z,0) |dy,
    \end{align*}
   and
    \begin{align*}
        III&=\int_{-\infty}^x \left| \frac{e^{2ik_{-,0}(x-y)}-1}{2ik_{-,0}}\,(q_{-}(y,\epsilon)-q_{-}(y,0))\,m_-(y,z,\epsilon) \right|dy\\
        &\lesssim \int_{-\infty}^x \left| (q_{-}(y,\epsilon)-q_{-}(y,0))\right|dy.
    \end{align*}
    Let
    $$
    G(x,k):=\frac{e^{2ikx}-1}{2ik},
    $$
    Then
    $$\partial_k G(x,k)= -\frac{e^{2ikx}-1}{2ik^2}+\frac{xe^{2ikx}}{k}.$$
    Therefore,
    \begin{align*}
        \left|\frac{e^{2ik_{-,\epsilon}(x-y)}-1}{2ik_{-,\epsilon}}-\frac{e^{2ik_{-,0}(x-y)}-1}{2ik_{-,0}} \right|\lesssim |k_{-,\epsilon}-k_{-,0}|\, \sup_{\im \theta\gtrsim 1} |\partial_k G(x-y,\theta)|\lesssim \epsilon (1+|x-y|).
    \end{align*}
    Hence
    \begin{align*}
        II&=\int_{-\infty}^x \left| \left(\frac{1-e^{2ik_{-,\epsilon}(x-y)}}{2ik_{-,\epsilon}}-\frac{1-e^{2ik_{-,0}(x-y)}}{2ik_{-,0}} \right) \,q_{-}(y,0)\,m_-(y,z,\epsilon)\right|dy\\
        &\lesssim \epsilon \int_{-\infty}^x  (1+|y|) \,|q_{-}(y,0)|dy.    
    \end{align*}
Combining these estimates, we obtain
\begin{align*}
    &|m_-(x,z,\epsilon)-m_-(x,z,0)|\\
    &\lesssim \int_{-\infty}^x \Bigl(\epsilon (1+|y|) \,|q_{-}(y,0)|
    +\left| q_{-}(y,\epsilon)-q_{-}(y,0)\right|
    +| q_{-}(y,0)|\,|m_-(y,z,\epsilon) -m_-(y,z,0) |\Bigr)\,dy.
\end{align*}
By Gronwall's inequality, we have
\begin{align*}
    |m_-(x,z,\epsilon)-m_-(x,z,0)|\lesssim e^{C\|q_-(\cdot,0)\|_{L^1(-\infty,0)}} \int_{-\infty}^x \Bigl(\epsilon (1+|y|) \,|q_{-}(y,0)|
    +\left| q_{-}(y,\epsilon)-q_{-}(y,0)\right|\Bigr)\,dy\lesssim \epsilon^{\frac{1}{2}}.
\end{align*}
Using similar estimates, keeping in mind that 
\begin{align*}
    \left|\frac{e^{2ik_+(y-x)}-1}{2ik_+}\right|\lesssim 1+ |y-x|
\end{align*}
holds uniformly on compact subsets of $\mathbb{C}\setminus(0,+\infty)$,
we also have
\begin{align*}
    |m_+(x,z,\epsilon)-m_+(x,z,0)|\lesssim e^{C\|(1+x) q_+(x,0)\|_{L^1(0,+\infty)}} \int_{0}^{+\infty} (1+y)\left| q_{+}(y,\epsilon)-q_{+}(y,0)\right|dy\lesssim \epsilon,
\end{align*}
  and
  \begin{align*}
    |m_+'(x,z,\epsilon)-m_+'(x,z,0)|\lesssim e^{C\| q_+(x,0)\|_{L^1(0,+\infty)}} \int_{0}^{+\infty} \left| q_{+}(y,\epsilon)-q_{+}(y,0)\right|dy\lesssim \epsilon,
\end{align*}
and
\begin{align*}
    |m_-'(x,z,\epsilon)-m_-'(x,z,0)|\lesssim e^{C\|q_-(\cdot,0)\|_{L^1(-\infty,0)}} \int_{-\infty}^x \Bigl(\epsilon (1+|y|) \,|q_{-}(y,0)|
    +\left| q_{-}(y,\epsilon)-q_{-}(y,0)\right|\Bigr)\,dy\lesssim \epsilon^{\frac{1}{2}}.
\end{align*}
As a result,
\begin{align}\label{error-estimates-W}
    |W(z, \epsilon)-W(z,0)|\lesssim \epsilon^{\frac{1}{2}}
\end{align}
holds uniformly on $K$.

We also seek $\tilde{f}_-(x,z,\epsilon)$ satisfying
\begin{align*}
    H_{\epsilon} \tilde{f}_-(x,z,\epsilon)=z\tilde{f}_-(x,z,\epsilon), \quad \tilde{f}_-(x,z,\epsilon)\to e^{ixk_{-,\epsilon}} \quad\text{as } x\to -\infty.
\end{align*}
Define
\begin{align*}
    \tilde{m}_-(x,z,\epsilon)=e^{-ixk_{-,\epsilon}}\tilde{f}_-(x,z,\epsilon),
\end{align*}
Then
\begin{align*}
    \tilde{m}_-(x,z,\epsilon)=1+\int_{-\infty}^{x}\frac{1-e^{2ik_{-,\epsilon}(x-y)}}{2ik_{-,\epsilon}}\,q_{-}(y,\epsilon)\,\tilde{m}_-(y,z,\epsilon)\,dy,
\end{align*}
Gronwall's inequality again gives
\begin{align*}
    |\tilde{m}_-(x,z,\epsilon)-1|\lesssim e^{C\|q_-(\cdot,\epsilon)\|_{L^1(-\infty,0)}} \int_{-\infty}^x  |q_{-}(y,\epsilon)| dy \lesssim e^{-c|x|}, \quad \text{for}~ x\le 0.
\end{align*}
Now assume
\begin{align*}
    f_+(x,z,\epsilon)=A(z,\epsilon)f_-(x,z,\epsilon)+B(z,\epsilon)\tilde{f}_-(x,z,\epsilon),
\end{align*}
then
\begin{align*}
    |A(z,\epsilon)|=\left|\frac{W[f_+(\cdot,z,\epsilon), \tilde{f}_-(\cdot,z,\epsilon)]}{W[f_-(\cdot,z,\epsilon),\tilde{f}_-(\cdot,z,\epsilon)]}\right|=\frac{\big|W[f_+(\cdot,z,\epsilon), \tilde{f}_-(\cdot,z,\epsilon)]\big|}{|2ik_{-,\epsilon}|}\lesssim 1,
\end{align*}
\begin{align*}
    B(z,\epsilon)=\frac{W[f_+(\cdot,z,\epsilon), f_-(\cdot,z,\epsilon)]}{W[\tilde{f}_-(\cdot,z,\epsilon),f_-(\cdot,z,\epsilon)]}=\frac{W(z,\epsilon)}{2ik_{-,\epsilon}},
\end{align*}
hence
\begin{align*}
    \left|e^{-ixk_{-,\epsilon}}f_+(x,z,\epsilon)-\frac{W(z,\epsilon)}{2ik_{-,\epsilon}}\right|
    &=\left|e^{-ixk_{-,\epsilon}}A(z,\epsilon)f_-(x,z,\epsilon)+\frac{W(z,\epsilon)}{2ik_{-,\epsilon}}(\tilde{m}_-(x,z,\epsilon)-1)\right|\\
    &\lesssim e^{-c|x|}, \quad \text{as }x\to -\infty.
\end{align*}

Let \(H_{\epsilon,R}=-\partial_r^2-3P_\epsilon^2(r-R)+5P_\epsilon^4(r-R)\) denote the self-adjoint Schr\"odinger operator on \(L^2(0,+\infty)\) with Dirichlet boundary conditions. For \(R>0\), consider the Jost solutions of \(H_{\epsilon,R}\):
\begin{align*}
    H_{\epsilon,R} f_+(x,z,\epsilon,R)=zf_+(x,z,\epsilon,R), \quad f_+(x,z,\epsilon,R)\to e^{ixk_{+}} \quad\text{as } x\to +\infty,
\end{align*}
Indeed,
\begin{align*}
    f_+(x,z,\epsilon,R)=f_+(x-R,z,\epsilon),
\end{align*}
In particular,
\begin{align}\label{error-estimates-fR}
    \left|e^{iRk_{-,\epsilon}}f_+(0,z,\epsilon,R)-\frac{W(z,\epsilon)}{2ik_{-,\epsilon}}\right| \lesssim e^{-cR}.
\end{align}
We now define the Jost solution of \(\widetilde H_\epsilon^{(0)}\):
\begin{align*}
    \widetilde H_\epsilon^{(0)} \tilde{f}_+(x,z,\epsilon)=z\tilde{f}_+(x,z,\epsilon), \quad \tilde{f}_+(x,z,\epsilon)\to e^{ixk_{+}} \quad\text{as }x\to+\infty,
\end{align*}
Choose \(R=R_\epsilon\). We aim to show
\begin{align}\label{error-estimates-tildef}
    \left|e^{iR_\epsilon k_{-,\epsilon}}f_+(0,z,\epsilon,R_\epsilon)-e^{iR_\epsilon k_{-,\epsilon}}\tilde{f}_+(0,z,\epsilon)\right|\lesssim\epsilon^{\frac{1}{2}}.
\end{align}
Once this is proved, we can define
$$f_\epsilon(z):=e^{iR_\epsilon k_{-,\epsilon}}\tilde{f}_+(0,z,\epsilon),$$
and
$$g(z):=\frac{W(z,0)}{2ik_{-,0}}.$$

By \eqref{error-estimates-W}, \eqref{error-estimates-fR}, and \eqref{error-estimates-tildef}, it follows that
the estimate \(|f_\epsilon(z)-g(z)|\lesssim\epsilon^{1/2}\) holds uniformly for \(z\in K\). Now define \(\Omega=\{z:|z|<100,\ \re z<0\}\) and set \(K=\partial\Omega\). Since \(H_0\) has only one negative eigenvalue and \(0\) is neither an eigenvalue nor a resonance of \(H_0\), the function \(g\) has exactly one zero in \(\Omega\) and does not vanish on \(K\). For sufficiently small \(\epsilon\), we have \(|f_\epsilon(z)-g(z)|<|g(z)|\) on \(K\). Both functions are analytic on \(\mathbb{C}\setminus(0,+\infty)\). By Rouch\'e's theorem, \(f_\epsilon\) and \(g\) have the same number of zeros in \(\Omega\). Thus \(\widetilde H_\epsilon^{(0)}\) has exactly one negative eigenvalue in \(\Omega\). Since \(\widetilde H_\epsilon^{(0)}>-10\), it has exactly one negative eigenvalue on \((-\infty,0]\). This proves the \(\ell=0\) case.

It remains to prove \eqref{error-estimates-tildef}. Define
\begin{align*}
H_{\epsilon,R_\epsilon}\psi(x,z,\epsilon)=z\psi(x,z,\epsilon), \quad\psi(0,z,\epsilon)=0, \quad\psi'(0,z,\epsilon)=1,
\end{align*}
and
\begin{align*}
\widetilde H_\epsilon^{(0)}\tilde{\psi}(x,z,\epsilon)=z\tilde{\psi}(x,z,\epsilon), \quad \tilde{\psi}(0,z,\epsilon)=0,\quad \tilde{\psi}'(0,z,\epsilon)=1,
\end{align*}
then $f_+(0,z,\epsilon,R_\epsilon)=W[\psi(\cdot,z,\epsilon),f_+(\cdot,z,\epsilon,R_\epsilon)]$ and $\tilde{f}_+(0,z,\epsilon)=W[\tilde{\psi}(\cdot,z,\epsilon), \tilde{f}_+(\cdot,z,\epsilon)]$, so it is equivalent to prove
\begin{align*}
    \left|e^{iR_\epsilon k_{-,\epsilon}}\left(W[\psi(\cdot,z,\epsilon),f_+(\cdot,z,\epsilon,R_\epsilon)]-W[\tilde{\psi}(\cdot,z,\epsilon), \tilde{f}_+(\cdot,z,\epsilon)]\right)\right|\lesssim\epsilon^{\frac{1}{2}}.
\end{align*}
It is then enough to prove
\begin{align*}
    &e^{-x \im k_{-,\epsilon}}\left(|\psi(x,z,\epsilon)|+|\psi'(x,z,\epsilon)|+|\tilde{\psi}(x,z,\epsilon)|+|\tilde{\psi}'(x,z,\epsilon)|\right)\lesssim 1, ~~~x\le R_\epsilon,\\
    &|f_+(x,z,\epsilon,R_\epsilon)|+|f_+'(x,z,\epsilon,R_\epsilon)|+|\tilde{f}_+(x,z,\epsilon)|+|\tilde{f}_+'(x,z,\epsilon)|\lesssim 1, ~~~x\ge R_\epsilon,
\end{align*}
and
\begin{align*}
    &e^{-x \im k_{-,\epsilon}}\left|\psi(x,z,\epsilon)-\tilde{\psi}(x,z,\epsilon)\right|\lesssim\epsilon^{\frac{1}{2}}, ~~~x\le R_\epsilon,\\
    &e^{-x \im k_{-,\epsilon}}\left|\psi'(x,z,\epsilon)-\tilde{\psi}'(x,z,\epsilon)\right|\lesssim\epsilon^{\frac{1}{2}}, ~~~x\le R_\epsilon,\\
    &\left|f_+(x,z,\epsilon,R_\epsilon)-\tilde{f}_+(x,z,\epsilon)\right|\lesssim\epsilon^{\frac{1}{2}}, ~~~ x\ge R_\epsilon,\\
    &\left|f_+'(x,z,\epsilon,R_\epsilon)-\tilde{f}_+'(x,z,\epsilon)\right|\lesssim\epsilon^{\frac{1}{2}}, ~~~ x\ge R_\epsilon.
\end{align*}
Once we prove these estimates, evaluating the Wronskian at $x=R_\epsilon$ gives the desired estimate. Define
$$\tilde{V}(x, \epsilon) = -3Q_\omega^2+5Q_\omega^4,$$
and
$$\tilde{q}_-(x, \epsilon)=\tilde{V}(x, \epsilon)-V_{-,\epsilon}.$$
    Then
\begin{align*}
    |\tilde{q}_-(x, \epsilon)-q_{-}(x-R_\epsilon,\epsilon)|\lesssim\min\{\epsilon, e^{-c|x-R_\epsilon|}\}.
\end{align*}
We have
\begin{align*}
    &\psi(x,z,\epsilon)=\frac{\sin (k_{-,\epsilon}x)}{k_{-,\epsilon}}+\int_0^x \frac{\sin (k_{-,\epsilon}(x-y))}{k_{-,\epsilon}}q_{-}(y-R_\epsilon,\epsilon)\psi(y,z,\epsilon)dy,\\
    &\tilde{\psi}(x,z,\epsilon)=\frac{\sin (k_{-,\epsilon}x)}{k_{-,\epsilon}}+\int_0^x \frac{\sin (k_{-,\epsilon}(x-y))}{k_{-,\epsilon}}\tilde{q}_-(y,\epsilon)\tilde{\psi}(y,z,\epsilon)dy.
\end{align*}
Since 
$$\left|\frac{\sin (k_{-,\epsilon}x)}{k_{-,\epsilon}}\right|\lesssim e^{\im k_{-,\epsilon}x},$$
by Gronwall's inequality we have
\begin{align*}
    |\psi(x,z,\epsilon)|+|\tilde{\psi}(x,z,\epsilon)|\lesssim e^{\im k_{-,\epsilon}x} e^{\int_0^x|q_{-}(y-R_\epsilon,\epsilon)|+|\tilde{q}_-(y,\epsilon)|dy}\lesssim e^{\im k_{-,\epsilon}x}, \quad \text{for} ~~x\le R_\epsilon.
\end{align*}
Similar estimates hold for \(\psi'(x,z,\epsilon)\) and \(\tilde{\psi}'(x,z,\epsilon)\). In addition,
\begin{align*}
    |\psi(x,z,\epsilon)-\tilde{\psi}(x,z,\epsilon)|
    \lesssim& \int_0^x e^{\im k_{-,\epsilon}(x-y)} |q_{-}(y-R_\epsilon,\epsilon)\psi(y,z,\epsilon)-\tilde{q}_-(y,\epsilon)\tilde{\psi}(y,z,\epsilon)|dy\\
    \lesssim& \int_0^x e^{\im k_{-,\epsilon}(x-y)} |q_{-}(y-R_\epsilon,\epsilon)-\tilde{q}_-(y,\epsilon)|\, |\tilde{\psi}(y,z,\epsilon)|dy\\
    &+\int_0^x e^{\im k_{-,\epsilon}(x-y)} |q_{-}(y-R_\epsilon,\epsilon)|\,|\psi(y,z,\epsilon)-\tilde{\psi}(y,z,\epsilon)|dy.
\end{align*}
This gives
\begin{align*}
 |\psi(x,z,\epsilon)-\tilde{\psi}(x,z,\epsilon)|
    \lesssim   e^{\im k_{-,\epsilon}x} e^{\int_0^x |q_{-}(y-R_\epsilon,\epsilon)|dy} \int_0^x |q_{-}(y-R_\epsilon,\epsilon)-\tilde{q}_-(y,\epsilon)|dy\lesssim e^{\im k_{-,\epsilon}x} \epsilon^{\frac{1}{2}}.
\end{align*}
A similar estimate holds for \(|\psi'(x,z,\epsilon)-\tilde{\psi}'(x,z,\epsilon)|\). Define
\begin{align*}
  &m_+(x,z,\epsilon,R_\epsilon)=e^{-ik_+x} f_+(x,z,\epsilon,R_\epsilon),\\
  &\tilde{m}_+(x,z,\epsilon)=e^{-ik_+x} \tilde{f}_+(x,z,\epsilon).
\end{align*}
We also have
\begin{align*}
&m_+(x,z,\epsilon,R_\epsilon)=1+\int_{x}^{\infty}\frac{e^{2ik_+(y-x)}-1}{2ik_+}\,V(y-R_\epsilon,\epsilon)\,m_+(y,z,\epsilon,R_\epsilon)dy,\\
&\tilde{m}_+(x,z,\epsilon)=1+\int_{x}^{\infty}\frac{e^{2ik_+(y-x)}-1}{2ik_+}\,\tilde{V}(y,\epsilon)\,\tilde{m}_+(y,z,\epsilon)dy,\\
\end{align*}
Since 
$$\left|\frac{e^{2ik_+(y-x)}-1}{2ik_+}\right|\lesssim 1+|y-R_\epsilon|, \quad \text{for}~~ y\ge x\ge R_\epsilon,$$
we have
\begin{align*}
    |m_+(x,z,\epsilon,R_\epsilon)|+|\tilde{m}_+(x,z,\epsilon)|\lesssim e^{\int_x^{+\infty}(1+|y-R_\epsilon|)(|V(y-R_\epsilon,\epsilon)|+|\tilde{V}(y,\epsilon)|) dy}\lesssim 1.
\end{align*}
The same holds for $m_+'(x,z,\epsilon,R_\epsilon)$ and $\tilde{m}_+'(x,z,\epsilon)$. Hence we have
\begin{align*}
    |f_+(x,z,\epsilon,R_\epsilon)|+|\tilde{f}_+(x,z,\epsilon)|+|f_+'(x,z,\epsilon,R_\epsilon)|+|\tilde{f}_+'(x,z,\epsilon)|\lesssim e^{-\im k_+ x}.
\end{align*}
In addition,
\begin{align*}
    |m_+(x,z,\epsilon,R_\epsilon)-\tilde{m}_+(x,z,\epsilon)|
    \lesssim& \int_{x}^{\infty} (1+|y-R_\epsilon|)|V(y-R_\epsilon,\epsilon)-\tilde{V}(y,\epsilon)|\, |\tilde{m}_+(y,z,\epsilon)|dy\\
    &+\int_{x}^{\infty} (1+|y-R_\epsilon|)|V(y-R_\epsilon,\epsilon)|\, |m_+(y,z,\epsilon,R_\epsilon)-\tilde{m}_+(y,z,\epsilon)|\,dy,
\end{align*}
this gives
\begin{align*}
 &|m_+(x,z,\epsilon,R_\epsilon)-\tilde{m}_+(x,z,\epsilon)|\\
    &\lesssim e^{\int_{x}^{\infty} (1+|y-R_\epsilon|)|V(y-R_\epsilon,\epsilon)|dy}  \int_{x}^{\infty} (1+|y-R_\epsilon|)|V(y-R_\epsilon,\epsilon)-\tilde{V}(y,\epsilon)|dy \\
    &\lesssim \epsilon^{\frac{1}{2}}.
\end{align*}
The same holds for $|m_+'(x,z,\epsilon,R_\epsilon)-\tilde{m}_+'(x,z,\epsilon)|$. Hence
we have
\begin{align*}
|f_+(x,z,\epsilon,R_\epsilon)-\tilde{f}_+(x,z,\epsilon)|+|f_+'(x,z,\epsilon,R_\epsilon)-\tilde{f}_+'(x,z,\epsilon)|\lesssim \epsilon^{1/2} e^{-\im k_+ x}.
\end{align*}
This completes the proof for \(\ell=0\).

\smallskip
\noindent
\emph{Step 2: induction on $\ell$.}
We prove the case \(\ell>0\) by induction. Suppose the assertion holds for \(\ell\). We prove the case \(\ell+1\) by contradiction. Assume that \(\widetilde H_\epsilon^{(\ell+1)}\) has two negative eigenvalues \(\lambda_1<\lambda_2<0\), with corresponding eigenfunctions \(f_1\) and \(f_2\). Since
\begin{align*}
    \widetilde H_\epsilon^{(\ell+1)}-\widetilde H_\epsilon^{(\ell)}=\frac{2(\ell+1)}{r^2}>0,
\end{align*}
\(\widetilde H_\epsilon^{(\ell)}\) must have a negative eigenvalue \(\lambda_3\), with corresponding eigenfunction \(f_3\). By the inductive assumption, \(\lambda_3\) is the only negative eigenvalue of \(\widetilde H_\epsilon^{(\ell)}\). Since \(f_1\) and \(f_2\) are linearly independent, we can choose \(c_1\) and \(c_2\) such that
$$
\langle c_1f_1+c_2f_2, f_3 \rangle_{L^2(0, \infty)}=0,
$$
with at least one of \(c_1\) and \(c_2\) nonzero. Since \(\lambda_3\) is the only negative eigenvalue of \(\widetilde H_\epsilon^{(\ell)}\), the operator \(\widetilde H_\epsilon^{(\ell)}\) is nonnegative on \(\{f_3\}^{\perp}\). Hence
\begin{align*}
    \langle \widetilde H_\epsilon^{(\ell)}(c_1f_1+c_2f_2), c_1f_1+c_2f_2\rangle_{L^2(0,\infty)}\ge 0.
\end{align*}
On the other hand, 
\begin{align*}
     &\langle \widetilde H_\epsilon^{(\ell)}(c_1f_1+c_2f_2), c_1f_1+c_2f_2\rangle_{L^2(0,\infty)}\\
     =&\left\langle \left(\widetilde H_\epsilon^{(\ell+1)}-\frac{2(\ell+1)}{r^2}\right)(c_1f_1+c_2f_2), c_1f_1+c_2f_2\right\rangle_{L^2(0,\infty)}\\
     =& \lambda_1 c_1^2 \|f_1\|_{L^2(0,\infty)}^2+\lambda_2 c_2^2 \|f_2\|_{L^2(0,\infty)}^2 - \int_0^\infty \frac{2(\ell+1)}{r^2} (c_1f_1+c_2f_2)^2\,dr
     <0,
\end{align*}
which is a contradiction. Therefore, every sector contains at most one eigenvalue below \(\omega\), as claimed.
\end{proof}

Since each \(L_+^{(\ell)}\) has at most one eigenvalue below \(\omega\), and \(L_+^{(\ell)}\) is strictly increasing in \(\ell\), there exists a maximal integer \(\ell_\epsilon\) such that \(L_+^{(\ell)}\) has an eigenvalue for \(0\leq\ell\leq\ell_\epsilon\), whereas \(L_+^{(\ell)}\) has no eigenvalue for \(\ell>\ell_\epsilon\). We next estimate \(\ell_\epsilon\).
\begin{prop}\label{prop:ell}
  We have \(\ell_\epsilon=\frac{3}{16}\epsilon^{-1}+\mathcal{O}(1)\).
\end{prop}
\begin{proof}
We prove separately the lower and upper bounds.

\smallskip
\noindent
\emph{Lower bound.}
Translation invariance gives $L_+^{(1)}Q_\omega'=0$, which after conjugation is equivalent to
\[
\widetilde H_\epsilon^{(1)}(rQ_\omega')=-\omega(rQ_\omega').
\]
Using $rQ_\omega'$ as a trial state in the $\ell$-th sector, we obtain
\begin{align*}
\left\langle \widetilde H_\epsilon^{(\ell)}(rQ_\omega'),rQ_\omega'\right\rangle_{L^2(0,\infty)}
&=\left\langle \left(\widetilde H_\epsilon^{(1)}+\frac{\ell(\ell+1)-2}{r^2}\right)(rQ_\omega'),rQ_\omega'\right\rangle \\
&=-\omega\int_0^\infty (rQ_\omega')^2\,dr
+\bigl(\ell^2+\ell-2\bigr)\int_0^\infty (Q_\omega')^2\,dr.
\end{align*}
If $\ell=\ell_\epsilon+1$, then by definition the operator $\widetilde H_\epsilon^{(\ell)}$ has no negative eigenvalue.  Therefore its quadratic form is nonnegative on every test function, in particular on $rQ_\omega'$.  We conclude that
\[
\ell_\epsilon^2+3\ell_\epsilon
\ge
\frac{\omega\int_0^\infty (rQ_\omega')^2\,dr}{\int_0^\infty (Q_\omega')^2\,dr}.
\]
We now estimate the ratio on the right-hand side. 
  The same argument as in the proof of Proposition~\ref{asymptotic-R} shows that
\[
\int_0^\infty (Q_\omega')^2\,dr
=\int_{\mathbb R}(P_0'(s))^2\,ds+\mathcal O(\epsilon)
=\frac{3\sqrt3}{64}+\mathcal O(\epsilon).
\]
Similarly,
\begin{align*}
\int_0^\infty (rQ_\omega')^2\,dr
&=\int_{|r-R_\epsilon|\le L_\epsilon} r^2(Q_\omega'(r))^2\,dr+\mathcal O(1)\\
&=R_\epsilon^2\int_{\mathbb R}(P_0'(s))^2\,ds+\mathcal O(R_\epsilon)\\
&=\frac{3\sqrt3}{64}R_\epsilon^2+\mathcal O(\epsilon^{-1}).
\end{align*}
Hence
\[
\frac{\omega\int_0^\infty (rQ_\omega')^2\,dr}{\int_0^\infty (Q_\omega')^2\,dr}
=\omega R_\epsilon^2+\mathcal O(\epsilon^{-1})
=\frac{9}{256}\epsilon^{-2}+\mathcal O(\epsilon^{-1}),
\]
where in the last step we used $\omega=\frac{3}{16}-\epsilon$ and Proposition~\ref{asymptotic-R}.  Solving the resulting quadratic inequality for $\ell_\epsilon$ gives
\[
\ell_\epsilon\ge \frac{3}{16}\epsilon^{-1}+\mathcal O(1).
\]

\smallskip
\noindent
\emph{Upper bound.}
Let $v$ be the normalized eigenfunction of $\widetilde H_\epsilon^{(\ell_\epsilon)}$ corresponding to its unique negative eigenvalue $-\lambda<0$:
\[
-\partial_r^2 v-3Q_\omega^2v+5Q_\omega^4v+\frac{\ell_\epsilon(\ell_\epsilon+1)}{r^2}v=-\lambda v,
\qquad \|v\|_{L^2(0,\infty)}=1.
\]
By the Sturm oscillation theorem, we may choose \(v>0\) on \((0,\infty)\). First, we prove the eigenfunction decays quickly. Note that \(\ell_\epsilon\geq\frac{3}{16}\epsilon^{-1}+\mathcal{O}(1)\), so at \(r=R_\epsilon\), we have \(\frac{\ell_\epsilon(\ell_\epsilon+1)}{r^2}\geq\frac{3}{16}+\mathcal{O}(\epsilon)\). On the other hand, since \(|Q_\omega(r)|\lesssim e^{-c|r-R_\epsilon|}\), we can choose a universal constant \(C\) such that, for \(r\geq R_\epsilon+C\),
$$|- 3Q_\omega^2 + 5Q_\omega^4|\le \frac{\ell_\epsilon(\ell_\epsilon+1)}{2r^2}.$$
Therefore on this region,
\[
v''(r)
=\left(-3Q_\omega^2+5Q_\omega^4+\frac{\ell_\epsilon(\ell_\epsilon+1)}{r^2}+\lambda\right)v(r)
\ge \frac{\ell_\epsilon(\ell_\epsilon+1)}{2r^2}v(r).
\]
Consider the comparison equation
\[
\tilde v''=\frac{\ell_\epsilon(\ell_\epsilon+1)}{2r^2}\tilde v,
\qquad \tilde v(r_0)=v(r_0),
\qquad \lim_{r\to\infty}\tilde v(r)=0,
\]
for some $r_0\ge R_\epsilon+C$.  Solving this Euler equation gives
\[
\tilde v(r)=v(r_0)\left(\frac{r}{r_0}\right)^{\frac{1-\sqrt{2\ell_\epsilon(\ell_\epsilon+1)+1}}{2}}.
\]
A comparison argument yields
\[
v(r)\le \tilde v(r),
\qquad r\ge r_0.
\]
This polynomial tail decay is in fact very strong because the exponent is of size $cR_\epsilon$.

Next we localize the quadratic form against the $\ell=1$ threshold.

  Let $\chi\ge 0$ be a smooth cut-off function, with $\chi(x)=1$ for $x\le 0$ and $\chi(x)=0$ for $x\ge 1$, and set
\[
\chi_\epsilon(r):=\chi\bigl(r-R_\epsilon-C-\sqrt{R_\epsilon}\bigr).
\]
Since $\widetilde H_\epsilon^{(1)}$ has the single negative eigenvalue $-\omega$ and no spectrum below it, one has
\[
\langle (\widetilde H_\epsilon^{(1)}+\omega)f,f\rangle\ge0
\qquad\text{for all }f\in H_0^1(0,\infty).
\]
Applying this to $f=\chi_\epsilon v$ gives
\begin{equation}\label{eq:chi-form}
\int_0^{\infty}(\partial_r(\chi_\epsilon v))^2+
\left(-3Q_\omega^2+5Q_\omega^4+\frac{2}{r^2}+\omega\right)(\chi_\epsilon v)^2\,dr\ge0.
\end{equation}
On the other hand, multiplying the eigenvalue equation for $v$ by $v$ and integrating by parts yields
\begin{equation}\label{eq:v-form}
\int_0^{\infty}(v')^2+
\left(-3Q_\omega^2+5Q_\omega^4+\frac{\ell_\epsilon(\ell_\epsilon+1)}{r^2}+\lambda\right)v^2\,dr=0.
\end{equation}
Subtracting \eqref{eq:chi-form} from \eqref{eq:v-form} and integrating by parts, we obtain
\begin{align*}
   \int_0^{+\infty} (1-\chi_\epsilon^2)(\partial_r v)^2+\left(- 3Q_\omega^2 + 5Q_\omega^4 + \frac{\ell_\epsilon(\ell_\epsilon+1)}{r^2}\right)(1-\chi_\epsilon^2)v^2 + \lambda v^2 \quad\quad\quad\quad\quad\quad  \\
+\left(\frac{\ell_\epsilon(\ell_\epsilon+1)-2}{r^2}-\omega\right)\chi_\epsilon^2v^2 + \chi_\epsilon \chi_\epsilon'' v^2 dr \le 0.
\end{align*}
Since on $r\ge R_\epsilon+C$, $|- 3Q_\omega^2 + 5Q_\omega^4|\le \frac{\ell_\epsilon(\ell_\epsilon+1)}{2r^2}$, we have
\begin{align*}
   \int_0^{+\infty}  \left(\frac{\ell_\epsilon(\ell_\epsilon+1)-2}{r^2}-\omega\right)\chi_\epsilon^2v^2 + \chi_\epsilon \chi_\epsilon'' v^2 dr \le 0,
\end{align*}
hence
\begin{align*}
    \ell_\epsilon(\ell_\epsilon+1)-2\le \frac{\int_0^{+\infty}  \omega\chi_\epsilon^2v^2 + |\chi_\epsilon \chi_\epsilon''| v^2 dr}{\int_0^{+\infty}  \frac{\chi_\epsilon^2v^2}{r^2} dr}.
\end{align*}
Note that since $\chi_\epsilon \chi_\epsilon''$ is supported on $[R_\epsilon+C+\sqrt{R_\epsilon}, R_\epsilon+C+\sqrt{R_\epsilon}+1]$, we have, for any $R_\epsilon + C \le r \le R_\epsilon + C+\sqrt{R_\epsilon}$,
\begin{align*}
    \int_0^{+\infty}|\chi_\epsilon \chi_\epsilon''v^2| dr &\lesssim v^2(R_\epsilon+C+1)\left(\frac{R_\epsilon+C+\sqrt{R_\epsilon}}{R_\epsilon+C+1}\right)^{1-\sqrt{2\ell_\epsilon(\ell_\epsilon+1)+1}}\\
    &\lesssim \int_{R_\epsilon+C}^{R_\epsilon+C+1} v^2(r) dr \, e^{-c\sqrt{R_\epsilon}}.
\end{align*}
The last inequality holds because 
$$\sqrt{2\ell_\epsilon(\ell_\epsilon+1)+1}\gtrsim \epsilon^{-1} $$
and
$$\ln \left(\frac{R_\epsilon+C+\sqrt{R_\epsilon}}{R_\epsilon+C+1}\right) \approx \epsilon \sqrt{R_\epsilon}.$$
Moreover,
\begin{align*}
    \int_{R_\epsilon+C}^{+\infty} \chi_\epsilon^2v^2 &= \sum_{k=0}^{+\infty} \int_{R_\epsilon+C+k}^{R_\epsilon+C+k+1} \chi_\epsilon^2v^2 dr := \sum_{k=0}^{+\infty} I_k.
\end{align*}
Then
\begin{align*}
    I_k\le \left(\frac{R_\epsilon+C+k+1}{R_\epsilon+C+1}\right)^{1-\sqrt{2\ell_\epsilon(\ell_\epsilon+1)+1}} I_0,
\end{align*}
and
\begin{align*}
    \int_{R_\epsilon+C}^{+\infty} \frac{\chi_\epsilon^2v^2}{r^2} &= \sum_{k=0}^{+\infty} \int_{R_\epsilon+C+k}^{R_\epsilon+C+k+1} \frac{\chi_\epsilon^2v^2}{r^2} dr \ge \sum_{k=0}^{+\infty} \frac{1}{(R_\epsilon+C+k+1)^2}I_k.
\end{align*}
Hence 
\begin{align*}
    &\int_{R_\epsilon+C}^{+\infty} \frac{\chi_\epsilon^2v^2}{r^2}dr-\frac{1}{(R_\epsilon+C+1)^2} \int_{R_\epsilon+C}^{+\infty} \chi_\epsilon^2v^2 dr\\ 
    \gtrsim& - \frac{1}{(R_\epsilon+C+1)^3} \sum_{k=0}^{+\infty} kI_k\\
    \gtrsim& - \frac{1}{(R_\epsilon+C+1)^3}\sum_{k=0}^{+\infty} k\left(\frac{R_\epsilon+C+k+1}{R_\epsilon+C+1}\right)^{1-\sqrt{2\ell_\epsilon(\ell_\epsilon+1)+1}} I_0\\
    \gtrsim& - \frac{1}{(R_\epsilon+C+1)^3} \int_{R_\epsilon+C}^{+\infty} \chi_\epsilon^2v^2 dr,
\end{align*}
or equivalently
\begin{align*}
     \int_{R_\epsilon+C}^{+\infty} \frac{\chi_\epsilon^2v^2}{r^2}dr \ge \frac{1+\mathcal{O}(\epsilon)}{(R_\epsilon+C+1)^2} \int_{R_\epsilon+C}^{+\infty} \chi_\epsilon^2v^2 dr. 
\end{align*}

Combining these two estimates, we have 
{\large\begin{align*}
   \displaystyle\frac{\int_0^{+\infty}  \omega\chi_\epsilon^2v^2 + |\chi_\epsilon \chi_\epsilon''| v^2 dr}{\int_0^{+\infty}  \frac{\chi_\epsilon^2v^2}{r^2} dr}
    =& \displaystyle \frac{\int_0^{R_\epsilon+C}  \omega\chi_\epsilon^2v^2  dr+ \int_{R_\epsilon+C}^{+\infty} \omega\chi_\epsilon^2v^2 + |\chi_\epsilon \chi_\epsilon''| v^2 dr}{\int_0^{R_\epsilon+C}  \frac{\chi_\epsilon^2v^2}{r^2} dr + \int_{R_\epsilon+C}^{+\infty}  \frac{\chi_\epsilon^2v^2}{r^2} dr}\\
    \le& \displaystyle \frac{\omega\int_0^{R_\epsilon+C}  \chi_\epsilon^2v^2  dr+ (\omega+e^{-c\sqrt{R_\epsilon}})\int_{R_\epsilon+C}^{+\infty} \chi_\epsilon^2v^2 dr}{\frac{1}{(R_\epsilon+C)^2}\int_0^{R_\epsilon+C}  \chi_\epsilon^2v^2 dr + \frac{1+\mathcal{O}(\epsilon)}{(R_\epsilon+C+1)^2} \int_{R_\epsilon+C}^{+\infty} \chi_\epsilon^2v^2 dr}\\
    \le& \, \omega R_\epsilon^2 + \mathcal{O}(\epsilon^{-1}).
\end{align*}}
Hence
\[
\ell_\epsilon(\ell_\epsilon+1)-2
\le \omega R_\epsilon^2+\mathcal O(\epsilon^{-1})
=\frac{9}{256}\epsilon^{-2}+\mathcal O(\epsilon^{-1}),
\]
which implies
\[
\ell_\epsilon\le \frac{3}{16}\epsilon^{-1}+\mathcal O(1).
\]
Together with the lower bound, this proves the proposition.
\end{proof}
Let \(\lambda_\ell\) denote the unique eigenvalue of \(L_+^{(\ell)}\). Then \(\lambda_\ell-\omega\) is the corresponding eigenvalue of \(\widetilde H_\epsilon^{(\ell)}\). Let \(v_\ell\) be the associated normalized eigenfunction. We first prove the following lemma.
\begin{lem}\label{formula-v}
    We have
    \begin{align*}
        \int_0^{+\infty} \frac{v_\ell^2}{r^2}dr= \frac{1}{R_\epsilon^2}+\mathcal{O}(\epsilon^3)
    \end{align*}
   \end{lem} 
\begin{proof}
Since $\|v_\ell\|_{L^2(0, +\infty)}=1$, it is enough to prove that
\[
\int_0^{\infty}\left(\frac1{r^2}-\frac1{R_\epsilon^2}\right)v_\ell(r)^2\,dr=\mathcal O(\epsilon^3).
\]
Multiplying by $R_\epsilon^2r^2$ and using $R_\epsilon\sim\epsilon^{-1}$, this is equivalent to showing
\begin{equation}\label{eq:weighted-r-minus-R}
\int_0^{\infty}\frac{(r-R_\epsilon)(r+R_\epsilon)}{r^2}v_\ell(r)^2\,dr=\mathcal O(\epsilon).
\end{equation}
Choose a sufficiently large constant \(C\) such that
$- 3Q_\omega^2  + 5Q_\omega^4\gtrsim 1$ for $r\le R_\epsilon-C$ and $|- 3Q_\omega^2  + 5Q_\omega^4 |\le \frac{1}{100} e^{-c|r-R_\epsilon|}$ for $r\ge R_\epsilon+C$ with some universal constant $c>0$. We divide the integral into three regions: the left tail $[0,R_\epsilon-C]$, the transition region $[R_\epsilon-C,R_\epsilon+C]$, and the right tail $[R_\epsilon+C,\infty)$.

On the transition strip one simply uses boundedness of the coefficient and normalization of $v_\ell$:
\[
\left|\int_{R_\epsilon-C}^{R_\epsilon+C}\frac{(r-R_\epsilon)(r+R_\epsilon)}{r^2}v_\ell(r)^2\,dr\right|
\lesssim \frac{C}{R_\epsilon}=\mathcal O(\epsilon).
\]
Thus it remains to estimate the two tails.

\smallskip
\noindent
\emph{Left tail.}
The eigenfunction solves
\[
-v_\ell''-3Q_\omega^2v_\ell+5Q_\omega^4v_\ell+\frac{\ell(\ell+1)}{r^2}v_\ell=(\lambda_\ell-\omega)v_\ell.
\]
On the left tail, the potential $-3Q_\omega^2+5Q_\omega^4$ is strictly positive because $Q_\omega$ stays close to $Q_\epsilon^*$.  Hence there exists $c_0>0$ such that for $r\le R_\epsilon-C$,
\[
v_\ell''(r)\ge c_0v_\ell(r).
\]
The eigenvalue equation and \(\|v_\ell\|_2=1\) also give
\(\|v_\ell'\|_2\le C\), uniformly in
\(0\le\ell\le\ell_\epsilon\); hence the one-dimensional Sobolev
inequality gives \(\|v_\ell\|_\infty\le C\).
For any $0\le r\le r_0\le R_\epsilon-C$, comparing with the explicit solution of $y''=c_0y$ that vanishes at the origin and matches $v_\ell$ at $r_0$, we obtain
\[
v_\ell(r)\le \frac{e^{\sqrt{c_0}r}-e^{-\sqrt{c_0}r}}{e^{\sqrt{c_0}r_0}-e^{-\sqrt{c_0}r_0}}v_\ell(r_0),
\qquad 0\le r\le r_0,
\]
for some $c>0$.  We deduce
\[
\left|\int_0^{R_\epsilon-C}\frac{(r-R_\epsilon)(r+R_\epsilon)}{r^2}v_\ell(r)^2\,dr\right|
\lesssim R_\epsilon \int_0^{R_\epsilon-C}\frac{e^{-c_0(R_\epsilon-r)}}{1+r^2}\,dr
\lesssim \frac1{R_\epsilon}=\mathcal O(\epsilon).
\]

\smallskip
\noindent
\emph{Right tail.}
Since \(\lambda_\ell-\omega\) is the only eigenvalue of \(\widetilde H_\epsilon^{(\ell)}\), we have
\begin{align*}
    &\left\langle \left(\widetilde H_\epsilon^{(\ell)}\right)(rQ_\omega'), rQ_\omega'\right\rangle_{L^2(0,\infty)}\\
    =&\left\langle \left(\widetilde H_\epsilon^{(1)}+\frac{\ell(\ell+1)-2}{r^2}\right)(rQ_\omega'), rQ_\omega'\right\rangle_{L^2(0,\infty)}\\
    =&-\omega\int_0^\infty (rQ_\omega')^2\,dr
    +(\ell^2+\ell-2)\int_0^\infty (Q_\omega')^2\,dr\\
    \ge&(\lambda_\ell-\omega)\int_0^\infty (rQ_\omega')^2\,dr.
\end{align*}
Hence
\[
\ell^2+\ell-2
\geq
\frac{\lambda_\ell\int_0^\infty(rQ_\omega')^2\,dr}
{\int_0^\infty(Q_\omega')^2\,dr}
=\lambda_\ell\left(R_\epsilon^2+\mathcal{O}(\epsilon^{-1})\right),
\]
where the last equality follows from the same argument as in Proposition~\ref{asymptotic-R}. Hence, for \(r\geq R_\epsilon+C\),
the following uniform lower bound holds after increasing the fixed
constant \(C\):
\[
 \partial_r^2v_\ell(r)\ge
 \frac{\omega R_\epsilon^2}{4r^2}v_\ell(r).
\]
For \(\ell=0\), this follows from
\(\omega-\lambda_0\ge\omega\) and the exterior decay of the potential.
For \(\ell\ge1\), monotonicity from the translation sector gives
\(0\le\lambda_\ell<\omega\), while the preceding Rayleigh inequality gives
\[
 \ell(\ell+1)\ge
 \lambda_\ell(R_\epsilon^2-CR_\epsilon)+2.
\]
Thus
\[
 \frac{\ell(\ell+1)}{r^2}+\omega-\lambda_\ell
 \ge \frac{\omega(R_\epsilon^2-CR_\epsilon)}{r^2},
\]
and the exponentially decaying negative part of
\(-3Q_\omega^2+5Q_\omega^4\) is absorbed by enlarging \(C\).
Comparing with the Euler equation  $y''=\frac{\omega R_\epsilon^2}{4r^2}y$, we deduce that for any $r\ge r_0\ge R_\epsilon+C$,
\[
v_\ell(r)\le v_\ell(r_0)\left(\frac{r_0}{r}\right)^{cR_\epsilon}
\]
for some universal constant $c>0$.  Consequently, if we set
\[
I_k:=\int_{R_\epsilon+C+k}^{R_\epsilon+C+k+1}|r-R_\epsilon|v_\ell(r)^2\,dr,
\]
then
\[
I_k\lesssim (k+1)\left(\frac{R_\epsilon+C+1}{R_\epsilon+C+k}\right)^{cR_\epsilon}
\int_{R_\epsilon+C}^{R_\epsilon+C+1}v_\ell(r)^2\,dr.
\]
Since the exponent $cR_\epsilon$ is large, 
\[
\sum_{k=0}^{\infty} I_k\lesssim 1,
\]
and hence
\begin{align*}
\left|\int_{R_\epsilon+C}^{\infty}\frac{(r-R_\epsilon)(r+R_\epsilon)}{r^2}v_\ell(r)^2\,dr\right|
&\lesssim \frac1{R_\epsilon}\int_{R_\epsilon+C}^{\infty}|r-R_\epsilon|v_\ell(r)^2\,dr\\
&\lesssim \frac1{R_\epsilon}\sum_{k=0}^{\infty}I_k
\lesssim \frac1{R_\epsilon}=\mathcal O(\epsilon).
\end{align*}
Combining the above estimates proves \eqref{eq:weighted-r-minus-R}, and therefore
\[
\int_0^{\infty}\frac{v_\ell(r)^2}{r^2}\,dr
=\frac1{R_\epsilon^2}+\mathcal O(\epsilon^3).
\]
\end{proof}

\begin{prop}\label{prop:eigenvaluegap}
    We have \(\lambda_{\ell}-\lambda_{\ell-1}=\frac{32}{3}\ell\epsilon^2+\mathcal{O}(\ell\epsilon^3)\).
\end{prop}
\begin{proof}
We compare neighboring sectors by the min--max principle.  Since
\[
\widetilde H_\epsilon^{(\ell)}=\widetilde H_\epsilon^{(\ell-1)}+\frac{2\ell}{r^2},
\]
we may use the normalized eigenfunction $v_{\ell-1}$ of $\widetilde H_\epsilon^{(\ell-1)}$ as a trial state for the $\ell$-th sector.  This gives
\begin{align*}
\lambda_{\ell}-\omega
&\le \langle \widetilde H_\epsilon^{(\ell)}v_{\ell-1},v_{\ell-1}\rangle \\
&=\left\langle \left(\widetilde H_\epsilon^{(\ell-1)}+\frac{2\ell}{r^2}\right)v_{\ell-1},v_{\ell-1}\right\rangle \\
&=(\lambda_{\ell-1}-\omega)+2\ell\int_0^{\infty}\frac{v_{\ell-1}(r)^2}{r^2}\,dr.
\end{align*}
Equivalently,
\begin{equation}\label{eq:upper-gap-bound}
\lambda_{\ell}-\lambda_{\ell-1}
\le 2\ell\int_0^{\infty}\frac{v_{\ell-1}(r)^2}{r^2}\,dr.
\end{equation}

Conversely, using the normalized eigenfunction $v_\ell$ of $\widetilde H_\epsilon^{(\ell)}$ as a trial state for the lowest eigenvalue in the $(\ell-1)$-st sector, we obtain
\begin{align*}
\lambda_{\ell-1}-\omega
&\le \langle \widetilde H_\epsilon^{(\ell-1)}v_\ell,v_\ell\rangle \\
&=\left\langle \left(\widetilde H_\epsilon^{(\ell)}-\frac{2\ell}{r^2}\right)v_\ell,v_\ell\right\rangle \\
&=(\lambda_\ell-\omega)-2\ell\int_0^{\infty}\frac{v_\ell(r)^2}{r^2}\,dr.
\end{align*}
Hence
\begin{equation}\label{eq:lower-gap-bound}
\lambda_{\ell}-\lambda_{\ell-1}
\ge 2\ell\int_0^{\infty}\frac{v_\ell(r)^2}{r^2}\,dr.
\end{equation}
Combining \eqref{eq:upper-gap-bound} and \eqref{eq:lower-gap-bound}, we arrive at
\[
2\ell\int_0^{\infty}\frac{v_\ell(r)^2}{r^2}\,dr
\le \lambda_{\ell}-\lambda_{\ell-1}
\le 2\ell\int_0^{\infty}\frac{v_{\ell-1}(r)^2}{r^2}\,dr.
\]
By Lemma~\ref{formula-v}, both integrals equal $R_\epsilon^{-2}+\mathcal O(\epsilon^3)$.  Therefore
\[
\lambda_{\ell}-\lambda_{\ell-1}
=\frac{2\ell}{R_\epsilon^2}+\mathcal O(\ell\epsilon^3).
\]
Finally, Proposition~\ref{asymptotic-R} gives
\[
\frac{2}{R_\epsilon^2}=\frac{32}{3}\epsilon^2+\mathcal O(\epsilon^3),
\]
which yields the claimed expansion.
\end{proof}
Summing the above relation from
\(2\) to \(\ell\) and using \(\lambda_1=0\) gives the following corollary.
\begin{cor}\label{prop:eigenvalueposition}
    \begin{align*}
        \lambda_\ell= \frac{16}{3}(\ell^2+\ell-2)\epsilon^2 + \mathcal{O}((\ell+1)^2 \epsilon^3), \quad \ell\ge 0.
        \end{align*}
\end{cor}

\begin{cor}[Multiplicity of the discrete eigenvalues]\label{prop:multiplicity}
Suppose $0\le \ell\le \ell_\epsilon$ and let $\lambda_\ell$ be the unique eigenvalue of $L_+^{(\ell)}$ in $(-\infty,\omega)$. Then the corresponding eigenspace of $L_+$ has dimension exactly $2\ell+1$.    
\end{cor}
\begin{proof}
By Proposition~\ref{prop:eigenvaluegap}, the eigenvalues \(\lambda_\ell\) are distinct. Let \(v_\ell\) be a nontrivial eigenfunction of \(H^{(\ell)}\) associated with \(\lambda_\ell\). For each \(m=-\ell,\dots,\ell\), define
\[
\Psi_{\ell m}(r,\theta):=\frac{v_\ell(r)}{r}Y_{\ell m}(\theta).
\]
These are \(2\ell+1\) linearly independent eigenfunctions of \(L_+\) associated with \(\lambda_\ell\).
 Since the eigenvalues belonging to distinct angular sectors are distinct by Proposition~\ref{prop:eigenvaluegap}, and each sector has at most one eigenvalue below \(\omega\), these functions exhaust the full eigenspace.
\end{proof}

\begin{proof}[Proof of Theorem~\ref{theorem-scalar}]
Proposition~3.1 shows that in each angular-momentum sector \(\ell\) there is
at most one eigenvalue below \(\omega\). Proposition \ref{prop:ell} estimates the maximal
sector index carrying a discrete eigenvalue:
\[
\ell_\epsilon=\frac{3}{16}\epsilon^{-1}+\mathcal O(1).
\]
Proposition \ref{prop:eigenvaluegap} and Corollary \ref{prop:eigenvalueposition} then estimate the eigenvalue gap
\[
\lambda_\ell-\lambda_{\ell-1}
=\frac{32}{3}\ell\epsilon^2+\mathcal O(\ell\epsilon^3)
\]
and gives
\[
\lambda_\ell=\frac{16}{3}(\ell^2+\ell-2)\epsilon^2
+\mathcal O((\ell+1)^2\epsilon^3).
\]
Corollary~\ref{prop:multiplicity} then shows that the corresponding
eigenspace of the full three-dimensional operator has dimension exactly
\(2\ell+1\). This proves the
theorem.
\end{proof}

\section{Internal modes of the matrix linearized Schr\"odinger operator}

We now study the internal modes of the full matrix operator obtained by linearizing the cubic--quintic nonlinear Schr\"odinger equation around the standing wave
\[
    u(t,x)=e^{i\omega t}Q_\omega(x).
\]
Recall that \(Q_\omega\) solves
\begin{equation*}
    -\Delta Q_\omega+\omega Q_\omega-Q_\omega^3+Q_\omega^5=0.
\end{equation*}
Let
\[
    u(t,x)=e^{i\omega t}\bigl(Q_\omega(x)+v(t,x)\bigr),
    \qquad
    v=a+ib,
\]
where \(a,b\) are real-valued. Linearizing
\[
    i\partial_tu=-\Delta u-|u|^2u+|u|^4u
\]
around \(Q_\omega\), we obtain
\begin{equation*}
    \partial_t
    \begin{pmatrix}
        a\\ b
    \end{pmatrix}
    =
    \mathcal L_\omega
    \begin{pmatrix}
        a\\ b
    \end{pmatrix},
\end{equation*}
where
\[
    \mathcal L_\omega=
    \begin{pmatrix}
        0&L_-\\
        -L_+&0
    \end{pmatrix},
\]
with
\begin{equation*}
    L_+
    =
    -\Delta+\omega-3Q_\omega^2+5Q_\omega^4,
    \qquad
    L_-
    =
    -\Delta+\omega-Q_\omega^2+Q_\omega^4 .
\end{equation*}

Since \(Q_\omega\) decays exponentially at spatial infinity, the essential spectrum of \(\mathcal L_\omega\) is determined by the limiting free operator
\[
    \mathcal L_{\omega,\infty}
    =
    \begin{pmatrix}
        0&-\Delta+\omega\\
        \Delta-\omega&0
    \end{pmatrix}.
\]
Therefore
\begin{equation*}
    \sigma_{\mathrm{ess}}(\mathcal L_\omega)
    =
    i(-\infty,-\omega]\cup i[\omega,\infty).
\end{equation*}

The internal modes of \(\mathcal L_\omega\) are nonzero eigenvalues
\[
    \mu=\pm i\nu,
    \qquad
    0<\nu<\omega,
\]
which can also be written in real form as follows. Suppose
\[
    \mathcal L_\omega
    \begin{pmatrix}
        f\\ g
    \end{pmatrix}
    =
    i\nu
    \begin{pmatrix}
        f\\ g
    \end{pmatrix}.
\]
Equivalently, after passing to a real basis, one obtains
\begin{equation}
    L_+ f=\nu g,
    \qquad
    L_- g=\nu f.
    \label{eq:real-internal-mode-system}
\end{equation}
Thus the internal mode problem is reduced to finding nontrivial real-valued pairs \((f,g)\) and frequencies \(\nu\in(0,\omega)\) satisfying \eqref{eq:real-internal-mode-system}.

Since \(Q_\omega\) is radial, the matrix operator \(\mathcal L_\omega\) decomposes according to spherical harmonics. Writing
\[
    f(r,\theta)
    =
    \sum_{\ell=0}^{\infty}
    \sum_{m=-\ell}^{\ell}
    f_{\ell m}(r)Y_{\ell m}(\theta),
\]
and
\[
    g(r,\theta)
    =
    \sum_{\ell=0}^{\infty}
    \sum_{m=-\ell}^{\ell}
    g_{\ell m}(r)Y_{\ell m}(\theta),
\]
we obtain a family of one-dimensional matrix operators
\begin{equation*}
   \mathcal L^{(\ell)}
    =
    \begin{pmatrix}
        0&L_-^{(\ell)}\\
        -L_+^{(\ell)}&0
    \end{pmatrix},
\end{equation*}
where
\begin{equation*}
    L_+^{(\ell)}
    =
    -\partial_r^2-\frac{2}{r}\partial_r
    +\omega-3Q_\omega^2+5Q_\omega^4
    +\frac{\ell(\ell+1)}{r^2},
\end{equation*}
and
\begin{equation*}
    L_-^{(\ell)}
    =
    -\partial_r^2-\frac{2}{r}\partial_r
    +\omega-Q_\omega^2+Q_\omega^4
    +\frac{\ell(\ell+1)}{r^2}.
\end{equation*}
Hence, it remains to find an internal mode in each angular-momentum sector \(\ell\) by solving
\begin{equation}
    L_+^{(\ell)}f=\nu g,
    \qquad
    L_-^{(\ell)}g=\nu f.
    \label{eq:angular-internal-mode}
\end{equation}

Define
\[
\begin{aligned}
    H_+^{(\ell)}
    &=-\partial_r^2+\omega-3Q_\omega^2+5Q_\omega^4
      +\frac{\ell(\ell+1)}{r^2},\\
    H_-^{(\ell)}
    &=-\partial_r^2+\omega-Q_\omega^2+Q_\omega^4
      +\frac{\ell(\ell+1)}{r^2},
\end{aligned}
\]
Then \(rL_\pm^{(\ell)}h=H_\pm^{(\ell)}(rh)\). Thus, upon setting \(F=rf\) and \(G=rg\), equation \eqref{eq:angular-internal-mode} becomes
\begin{equation}
    H_+^{(\ell)}F=\nu G,
    \qquad
    H_-^{(\ell)}G=\nu F.
    \label{eq:H-internal-mode}
\end{equation}

By gauge invariance,
\begin{equation*}
    H_-^{(0)}(rQ_\omega)=0.
\end{equation*}
Since \(Q_\omega>0\), the ground-state representation gives
\[
    H_-^{(\ell)}\geq0,
    \qquad
    \ker H_-^{(0)}=\operatorname{span}\{rQ_\omega\},
\]
and \(H_-^{(\ell)}>0\) for every \(\ell\geq1\). 

For $\ell\ge1$, $H_-^{(\ell)}$ is strictly positive; for $\ell=0$ we
work on $\{rQ_\omega\}^{\perp}$, where the restriction of
$H_-^{(0)}$ is strictly positive.  On the corresponding Hilbert space,
define $K^{(\ell)}$ as the self-adjoint operator associated with the
closed, lower-bounded quadratic form
\begin{equation}\label{eq:K-form-definition}
 k_\ell[h]
 :=\left\langle H_+^{(\ell)}(H_-^{(\ell)})^{1/2}h,
                    (H_-^{(\ell)})^{1/2}h\right\rangle,
 \qquad
 K^{(\ell)}=(H_-^{(\ell)})^{1/2}H_+^{(\ell)}
             (H_-^{(\ell)})^{1/2}.
\end{equation}
Here the displayed product is understood in this form sense; this removes
any ambiguity about domains of the unbounded factors.

There is a one-to-one correspondence between the positive eigenvalues of
\(K^{(\ell)}\) below \(\omega^2\) and the internal-mode eigenvalues of
\(\mathcal L^{(\ell)}\). Indeed, suppose
\[
    K^{(\ell)}h=\nu^2h,
    \qquad 0<\nu<\omega.
\]
Set
\[
    F=\bigl(H_-^{(\ell)}\bigr)^{1/2}h,
    \qquad
    G=\nu^{-1}H_+^{(\ell)}F.
\]
Then \(H_+^{(\ell)}F=\nu G\) and \(H_-^{(\ell)}G=\nu F\).
Consequently, with \(f=F/r\) and \(g=G/r\),
\[
    \mathcal L^{(\ell)}
    \binom{f}{\pm i g}
    =
    \pm i\nu\binom{f}{\pm i g}.
\]
Conversely, any eigenvector of \(\mathcal L^{(\ell)}\) with eigenvalue
\(\pm i\nu\), \(0<\nu<\omega\), yields a solution \((F,G)=(rf,rg)\) of
\eqref{eq:H-internal-mode}, and
\[
    h=\bigl(H_-^{(\ell)}\bigr)^{-1/2}F
\]
satisfies \(K^{(\ell)}h=\nu^2h\). These inverse constructions preserve
the eigenspace dimension. Thus an eigenvalue
\(\mu\in(0,\omega^2)\) of \(K^{(\ell)}\) corresponds precisely to the
pair \(\pm i\sqrt{\mu}\) of eigenvalues of \(\mathcal L^{(\ell)}\). We are now ready to prove Theorem \ref{theorem-matrix}:
\begin{proof}
\noindent
We separate our proof in several steps.

\medskip
\emph{Step 1.  Monotonicity in \(\ell\).}
By the conjugation, we only need to consider the positive eigenvalues of $K^{(\ell)}$. We first prove certain monotonicity in $\ell$. Since both potentials converge exponentially to zero as \(r\to\infty\),
\(K^{(\ell)}\) is a relatively compact perturbation of
\[
    \left(
        -\partial_r^2
        +\frac{\ell(\ell+1)}{r^2}+\omega
    \right)^2.
\]
Consequently,
\[
    \sigma_{\mathrm{ess}}\bigl(K^{(\ell)}\bigr)
    =
    [\omega^2,\infty).
\]

For the monotonicity argument below we take \(\ell\ge1\), so that
\(H_-^{(\ell)}\) is strictly positive.  The radial constrained sector will
be treated separately.

Since
$H_-^{(\ell)}\ge c_\ell>0$ for $\ell\ge1$, its inverse is bounded, and
congruence by $(H_-^{(\ell)})^{1/2}$ preserves the dimension of every
finite-dimensional negative subspace.  

We have the factorization
\[
\begin{aligned}
    K^{(\ell)}-\omega^2
    &=
    \bigl(H_-^{(\ell)}\bigr)^{1/2}
    \left(
        H_+^{(\ell)}
        -
        \omega^2\bigl(H_-^{(\ell)}\bigr)^{-1}
    \right)
    \bigl(H_-^{(\ell)}\bigr)^{1/2}.
\end{aligned}
\]
Let \(n_-(\cdot)\) denote the dimension of the negative spectral subspace of an operator. By the min--max principle and the invariance of the negative Morse index under congruence by
\(\bigl(H_-^{(\ell)}\bigr)^{1/2}\),
\[
    n_-\big(K^{(\ell)}-\omega^2\big)
    =
    n_-\big(H_+^{(\ell)}
        -
        \omega^2\bigl(H_-^{(\ell)}\bigr)^{-1}\big).
\]
Since
\[
    H_-^{(\ell+1)}
    =
    H_-^{(\ell)}
    +
    \frac{2\ell+2}{r^2},
\]
and
\[
    H_+^{(\ell+1)}
    =
    H_+^{(\ell)}
    +
    \frac{2\ell+2}{r^2},
\]
we have
\((H_-^{(\ell+1)})^{-1}\le(H_-^{(\ell)})^{-1}\).  Together with the corresponding increase of \(H_+^{(\ell)}\), this gives
$$
 H_+^{(\ell+1)}
        -
        \omega^2\bigl(H_-^{(\ell+1)}\bigr)^{-1}\ge H_+^{(\ell)}
        -
        \omega^2\bigl(H_-^{(\ell)}\bigr)^{-1}.
$$
Hence
\[
    n_-\big(H_+^{(\ell+1)}
        -
        \omega^2\bigl(H_-^{(\ell+1)}\bigr)^{-1}\big)
    \leq
    n_-\big(H_+^{(\ell)}
        -
        \omega^2\bigl(H_-^{(\ell)}\bigr)^{-1}\big),
\]
equivalently,
\[
    n_-\bigl(K^{(\ell+1)}-\omega^2\bigr)
    \leq
    n_-\bigl(K^{(\ell)}-\omega^2\bigr).
\]

\medskip

\noindent
\emph{Step 2. Existence for \(\ell\leq c_0 R_\epsilon\).} We next prove the existence of internal modes. Fix a finite-dimensional subspace
\begin{align*}
    E\subset C_c^\infty\left(\frac14,\frac12\right),
    \qquad
    \dim E=N,
\end{align*}
where \(N\geq3\) is fixed. Define the rescaled space
\begin{align*}
    E_\epsilon
    :=
    \left\{
        u(r)=R_\epsilon^{-1/2}
        \phi\left(\frac{r}{R_\epsilon}\right):
        \phi\in E
    \right\}.
\end{align*}
Every \(u\in E_\epsilon\) is supported in
\begin{align*}
    \frac14R_\epsilon
    \leq r\leq
    \frac12R_\epsilon.
\end{align*}

Inside this region, \(Q_\omega\) is exponentially close to the
flat-top equilibrium \(Q_\epsilon^*\), which satisfies
\begin{align*}
    \omega-(Q_\epsilon^*)^2+(Q_\epsilon^*)^4=0.
\end{align*}
Thus
\begin{align*}
    \sup_{\frac14R_\epsilon\leq r\leq\frac12R_\epsilon}
    \left|
        \omega-Q_\omega(r)^2+Q_\omega(r)^4
    \right|
    \leq
    Ce^{-cR_\epsilon}.
\end{align*}

By finite-dimensional norm equivalence and scaling, there exists a constant \(C\) (may depend on $E$), such that
\begin{align*}
    \int_0^\infty |u'(r)|^2\,dr
    \leq
    CR_\epsilon^{-2}\|u\|_{L^2}^2,
\end{align*}
and, since \(r\approx R_\epsilon\) on the support of \(u\),
\begin{align*}
    \ell(\ell+1)
    \int_0^\infty \frac{|u(r)|^2}{r^2}\,dr
    \leq
    C\frac{(\ell+1)^2}{R_\epsilon^2}
    \|u\|_{L^2}^2.
\end{align*}
Combining the above estimates, we obtain
\begin{align*}
    \left\langle H_-^{(\ell)}u,u\right\rangle_{L^2}
    \leq
    C\frac{(\ell+1)^2}{R_\epsilon^2}
    \|u\|_{L^2}^2.
\end{align*}

First suppose \(\ell\ge1\). In addition, the Cauchy--Schwarz inequality
gives
\begin{align*}
    \|u\|_{L^2}^4
    \leq
    \left\langle H_-^{(\ell)}u,u\right\rangle_{L^2}
    \left\langle
        \bigl(H_-^{(\ell)}\bigr)^{-1}u,u
    \right\rangle_{L^2}.
\end{align*}
Therefore,
\begin{align*}
    \left\langle
        \bigl(H_-^{(\ell)}\bigr)^{-1}u,u
    \right\rangle_{L^2}
    \geq
    c\frac{R_\epsilon^2}{(\ell+1)^2}
    \|u\|_{L^2}^2.
\end{align*}

The coefficients of \(H_+^{(\ell)}\) are uniformly bounded on the
support of \(u\). Hence
\begin{align*}
    \left\langle H_+^{(\ell)}u,u\right\rangle_{L^2}
    \leq
    C\left(
        1+\frac{(\ell+1)^2}{R_\epsilon^2}
    \right)
    \|u\|_{L^2}^2.
\end{align*}
Consequently,
\begin{align*}
    \left\langle
        \left(
            H_+^{(\ell)}
            -\omega^2\bigl(H_-^{(\ell)}\bigr)^{-1}
        \right)u,u
    \right\rangle_{L^2}
    &\leq
    \left[
        C\left(
            1+\frac{(\ell+1)^2}{R_\epsilon^2}
        \right)
        -
        c\omega^2
        \frac{R_\epsilon^2}{(\ell+1)^2}
    \right]
    \|u\|_{L^2}^2.
\end{align*}

Thus, if \(\ell\leq c_0R_\epsilon\) for some sufficiently small \(c_0\), then
\begin{align*}
    \left\langle
        \left(
            H_+^{(\ell)}
            -\omega^2\bigl(H_-^{(\ell)}\bigr)^{-1}
        \right)u,u
    \right\rangle_{L^2}
    <0
\end{align*}
for every \(0\neq u\in E_\epsilon\). Thus
\begin{align*}
    n_-\big(K^{(\ell)}-\omega^2\big)
    =n_-\left(
        H_+^{(\ell)}
        -\omega^2\bigl(H_-^{(\ell)}\bigr)^{-1}
    \right)\geq N,
    \qquad
    1\leq\ell\leq cR_\epsilon.
\end{align*}
By the scalar spectral result of Section~3, \(K^{(1)}\geq0\) has only the translation zero mode, whereas \(K^{(\ell)}>0\) for \(\ell\ge2\).  Since the preceding negative-index lower bound has dimension at least \(N\ge3\), \(K^{(\ell)}\) has at least one eigenvalue in \((0,\omega^2)\) for every \(1\le\ell\le cR_\epsilon\).

For \(\ell=0\), we restrict to \(\{rQ_\omega\}^\perp\). Since
\begin{align*}
    \dim E_\epsilon=N\geq3,
\end{align*}
one has
\begin{align*}
    \dim\bigl(E_\epsilon\cap\{rQ_\omega\}^\perp\bigr)\geq N-1\geq2.
\end{align*}
On this subspace, \(H_-^{(0)}>0\), and the same estimates show that
\(K^{(0)}-\omega^2\) has a negative subspace of dimension at least
\(N-1\ge2\).  Congruence with \(H_-^{(0)}\) shows that \(K^{(0)}\) has at
most the single negative direction of \(H_+^{(0)}\); moreover it has no zero
mode on the constrained radial space.  Hence at least one of these
eigenvalues lies in \((0,\omega^2)\).

\medskip

\noindent
\emph{Step 3. Absence for \(\ell\geq C_0 R_\epsilon\).}
Choose \(C_0>0\) sufficiently large. The plateau, interface, and exterior estimates of Section~2 imply
\[
 \sup_{r>0}r^2\left(
 |-Q_\omega(r)^2+Q_\omega(r)^4|
 +|-3Q_\omega(r)^2+5Q_\omega(r)^4|
 \right)\le CR_\epsilon^2.
\]
Indeed, the potentials are uniformly bounded for \(r\lesssim R_\epsilon\)
and decay exponentially for \(r\gtrsim R_\epsilon\).  Therefore, for \(\ell\geq C_0 R_\epsilon\),
\begin{align*}
	\frac{\ell(\ell+1)}{r^2} \ge \big|-Q_\omega(r)^2+Q_\omega(r)^4\big|+\big|-3Q_\omega(r)^2+5Q_\omega(r)^4\big|.
\end{align*}
Therefore,
\[
    H_-^{(\ell)}
    \geq
    -\partial_r^2+\omega
    \geq
    \omega
\]
and
\[
    H_+^{(\ell)}
    \geq
    -\partial_r^2+\omega
    \geq
    \omega.
\]
Consequently,
\begin{align*}
    K^{(\ell)}
    &=
    \bigl(H_-^{(\ell)}\bigr)^{1/2}
    H_+^{(\ell)}
    \bigl(H_-^{(\ell)}\bigr)^{1/2}
    \\
    &\geq
    \omega H_-^{(\ell)}
    \\
    &\geq
    \omega^2.
\end{align*}
Hence
\[
    \sigma\bigl(K^{(\ell)}\bigr)
    \cap(-\infty,\omega^2)
    =
    \varnothing.
\]

Combining Steps~1--3, there exists a unique integer \(\widetilde\ell_0\) such that
$$
cR_\epsilon \le \widetilde{\ell}_0 \le CR_\epsilon
$$
with the stated property.
\end{proof}






\bigskip

\end{document}